\documentclass[12pt,reqno]{article}
\usepackage{amsmath,amsfonts,amssymb,amscd,amsthm,amsbsy, epsf,calc,enumerate,cite}
\usepackage{color}
\usepackage{graphicx,gensymb}
\usepackage{epsfig,esint}
\usepackage{verbatim}
\usepackage{soul}
\definecolor{purple}{rgb}{0.65, 0, 1}
\definecolor{pink}{rgb}{1, 0, 0.9}
\definecolor{orange}{rgb}{1,.5,0}

\numberwithin{equation}{section}
\newtheorem{theorem}{Theorem}[section]

\newtheorem{lemma}[theorem]{Lemma}
\newtheorem{definition}[theorem]{Definition}
\newtheorem{proposition}[theorem]{Proposition}

\newcommand{\comments}[1]{}

\newcounter{DTheoremCounter}

\newcounter{DAppCounter}
\newcounter{claimCounter}
\newcommand{\commentout}[1]{{}}

\definecolor{purple}{rgb}{0.65, 0, 1}
\definecolor{green}{rgb}{0, 1, 0}
\definecolor{orange}{rgb}{1,.5,0}

\newcommand{\W}{{w}}
\newcommand{\V}{{v}}
\newcommand{\tv}{{\tilde v}}
\def \Rm {\mathbb R}
\def \Tm {\mathbb T}

\def \eps {\epsilon}

\newcommand{\vertiii}[1]{{\vert\kern-0.25ex\vert\kern-0.25ex\vert #1
    \vert\kern-0.25ex\vert\kern-0.25ex\vert}}

\begin{document}

\title{Finite time singularity formation \\ for the modified SQG patch equation}
\author{Alexander Kiselev\footnote{Department of Mathematics, Rice
    University, 6100 S. Main St.,
Houston, TX 77005-1892; kiselev@rice.edu}
\and Lenya Ryzhik\footnote{Department of Mathematics, Stanford
  University, Stanford, CA 94305; ryzhik@math.stanford.edu}
\and Yao Yao\footnote{School of Mathematics, Georgia Institute of Technology, 686 Cherry Street, Atlanta, GA 30332-0160; yaoyao@math.gatech.edu}
\and Andrej Zlato\v s\footnote{Department of Mathematics, University of
  Wisconsin, 480 Lincoln Dr.
Madison, WI 53706-1325; zlatos@math.wisc.edu}}

\maketitle

\begin{abstract}
It is well known that the incompressible Euler equations in two dimensions
have globally regular solutions.
The inviscid surface quasi-geostrophic (SQG) equation has a Biot-Savart law which
is one derivative less regular than
in the Euler case, and the question of global regularity
for its solutions is still open.
We study here the patch dynamics in the half-plane for a
family of active scalars which interpolates between these two equations, via
a parameter~$\alpha\in[0,\frac 12]$ appearing
in the kernels of their Biot-Savart laws.
The values~$\alpha=0$ and~$\alpha=\frac 12$ correspond to the 2D Euler and  SQG cases,
respectively. We prove global in time regularity for the 2D Euler patch model, even if the patches initially touch the boundary of the half-plane.
On the other hand,  for any sufficiently small $\alpha>0$,
we  exhibit initial data which lead to a singularity in  finite time.
Thus, these results show a phase transition in the behavior of solutions to these equations, and provide a rigorous  foundation for classifying the 2D Euler equations as critical.
\end{abstract}

\section{Introduction}

The question of global regularity of  solutions is still open
for many fundamental equations of fluid dynamics. In the case of the three dimensional
Navier-Stokes and Euler equations, it remains one of the central open problems
of classical mathematical physics and partial differential equations.
Much more is known in two dimensions, though the picture is far from
complete even in that case.
Global regularity of solutions to the 2D incompressible Euler equations in
smooth domains has been known since the works of
Wolibner \cite{Wolibner} and H\"older \cite{Holder}.
However, even in 2D the estimates necessary for the Euler global regularity barely close,
and the best upper bound on the growth of derivatives is double exponential
in time. Recently,  Kiselev and \v Sver\' ak showed that this upper bound is sharp by constructing an example  of a solution
to the 2D Euler equations on a disk whose gradient indeed grows
double exponentially in time \cite{KS}.  Exponential growth on a domain without a boundary (the torus $\Tm^2$)
was recently shown to be possible by Zlato\v s \cite{ZlaEuler}.
Some earlier examples of unbounded growth are due to Yudovich \cite{Yud1,Yud2},
Nadirashvili \cite{Nad}, and Denisov \cite{Den1,Den2}.
In a certain sense that will be made precise below, the 2D Euler equations
may be regarded as critical,
even though we are not aware of a simple scaling argument for such a classification.

\subsubsection*{The SQG and modified SQG equations}

As opposed to the 2D Euler equations, the global regularity vs finite time singularity question
has not been resolved for the two-dimensional
surface quasi-geostrophic (SQG) equation, which appears in atmospheric science models (and shares many of its features with the 3D Euler equation --- see, e.g., \cite{CMT,mb,Rodrigo}).
The SQG equation is given by
\begin{equation}\label{sqg}
\partial_t \omega + (u \cdot \nabla) \omega =0,
\end{equation}
with $\omega(\cdot,0)=\omega_0$
 and the Biot-Savart law for the velocity
\[
u := \nabla^\perp (-\Delta)^{-1/2} \omega,
\]
where $\nabla^\perp := (\partial_{x_2}, -\partial_{x_1}).$
Equation (\ref{sqg}) has the
same form as the 2D Euler equations in the vorticity formulation,
but the latter has the more regular (by one derivative) Biot-Savart law
\[
u := \nabla^\perp (-\Delta)^{-1} \omega.
\]
The SQG equation is usually considered on either $\Rm^2$ or $\Tm^2$, and
the fractional Laplacian can be defined via the Fourier transform.
The equation appears, for instance, in the book \cite{Ped} by Pedlosky  and was first rigorously studied  in the work of
Constantin, Majda and Tabak \cite{CMT} where, in particular,
a closing saddle scenario for a finite time singularity has been suggested.
This scenario and some other related ones have been ruled out in
the later works of C\' ordoba \cite{Cord1} and C\' ordoba and Fefferman \cite{CoFe}.
Also, existence of global weak solutions   was proved by Resnick \cite{Resnick}.

We should mention that a lot of work has focused on the SQG equation and
related active scalars with a fractional dissipation
term of the form $-(-\Delta)^\beta \omega$ on
the right-hand side of~\eqref{sqg}. Global regularity for the critical viscous SQG equation, with $\beta=\frac 12$, was proved
independently by Caffarelli and Vasseur \cite{CV}, and by Kiselev, Nazarov, and Volberg \cite{KNV}
(see also the subsequent works \cite{KN,CoVi,CVT} for alternative proofs).
The global regularity proof is standard for~$\beta\in(\frac 12,1]$ (see e.g. \cite{NJu}), while in the super-critical case~$\beta <\frac 12$
the question of global regularity vs finite time blow-up remains open. The best available result in this direction
is global regularity for the logarithmically super-critical SQG equation by  Dabkowski, Kiselev, Silvestre, and Vicol \cite{DKSV}.

A natural family of active scalars which interpolates between the 2D Euler
and SQG equations is given by \eqref{sqg} with the Biot-Savart law
\[
u := \nabla^\perp (-\Delta)^{-1+\alpha}\omega.
\]
This family  has been called modified or generalized SQG equations in the literature
(see, e.g., \cite{CIW}, or the paper \cite{PHS} by Pierrehumbert, Held, and Swanson  for a geophysical literature reference). The cases $\alpha=0$ and $\alpha=\frac 12$ correspond to
the 2D Euler and SQG equations, respectively.   The question of global regularity
of the solutions with smooth initial data has been open for all $\alpha>0$,
that is, for any of these models which are more singular than the 2D Euler equations.  Ironically, even though the
SQG and the modified~SQG equations are more singular than the 2D Euler equations,
no examples of solutions with unbounded growth of derivatives in time are known.
The best result in this direction is arbitrary bounded growth of high Sobolev
norms on finite time intervals by Kiselev and Nazarov \cite{KN10}.
The reason is that due to nonlinearity and nonlocality of active scalars,
it is difficult to   control the solutions
at large  times, and this task gets harder as the Biot-Savart law becomes more singular.
This issue will be evident in the present paper as well.

 \subsubsection*{Vortex patches}

While the above discussion concerns active scalars with sufficiently smooth initial data, an important class
of solutions to  these equations are vortex patches
\begin{equation*}
\omega(x,t)=\sum_k \theta_k\chi_{\Omega_k(t)}(x).  
\end{equation*}
Here $\theta_j$ are some constants, $\Omega_j(t)$ are (evolving in time)
open sets with non-zero mutual distances and smooth boundaries, and $\chi_D$ denotes the characteristic function
of a domain $D$. Vortex patches model flows with abrupt variations in vorticity, which are common in nature.
Existence and uniqueness of appropriately defined vortex patch solutions to the 2D Euler equations in the whole plane
goes back to the work of  Yudovich \cite{Yudth}, and
regularity in this setting refers to sufficient smoothness of the patch boundaries as well as to the
lack of both self-intersections of each patch boundary and touches of different patches.

Singularity formation for 2D Euler patches had initially been conjectured  based on  the
numerical simulations by Buttke \cite{Butt}, see Majda \cite{Majda} for a discussion.  Later, simulations by Dritschel, McIntyre, and Zabusky \cite{DM,DZ}
questioning the singularity formation prediction appeared;  we refer to \cite{Pulli} for a review of these and related works.
This controversy was settled in 1993, when  Chemin \cite{c} proved that the boundary of a 2D Euler patch remains regular for all times,
with a double exponential upper bound on the temporal growth of its curvature
(see also the work by Bertozzi and Constantin \cite{bc} for a different proof).

The patch problem for the SQG equation is more involved.
Local existence and uniqueness in the class of weak solutions of the special type
\[
\omega(\cdot,t)=\chi_{\{x_2<\varphi(x_1,t)\}},
\]
with $\varphi\in C^\infty$ and periodic in $x_1$, corresponding to a (single patch) initial condition of the same form,
was proved by Rodrigo \cite{Rodrigo}.
For the SQG and modified SQG patches with boundaries which are simple closed $H^3$ curves, local existence was established by
Gancedo~\cite{g} via a study of a contour equation whose solutions parametrize the patch boundary (uniqueness of solutions was also
proved for the contour equation for~$\alpha\in(0,\frac 12)$, although not for the original modified SQG equation).  Local existence of
such contour solutions in the more singular case $\alpha\in(\frac 12,1]$ was obtained by Chae, Constantin, C\' ordoba, Gancedo, and
Wu \cite{CCCGW}.  Existence of splash singularities (touching of exactly  two segments of a patch boundary, which
remains uniformly $H^3$) for the SQG equation was ruled out by Gancedo and Strain \cite{GS}.

A computational study of the SQG and modified SQG patches by C\' ordoba, Fontelos, Mancho, and
Rodrigo \cite{CFMR} (where the patch problem for the modified SQG equation first appeared) suggested a
finite time singularity, with two patches touching each other and simultaneously developing corners at the touching
point. A more careful numerical study
by Mancho \cite{Mancho} suggests involvement of self-similar elements in this singularity formation
process, but its rigorous confirmation and understanding is still lacking.
We note that even local well-posedness is far from trivial for many interface evolution models of fluid dynamics,
see e.g. \cite{CCG} where the Muskat problem is discussed.
We refer to~\cite{CCFGL,CCFGG,Wu1,Wu2} for other recent advances in some of the interface problems of fluid dynamics.

\subsubsection*{Vortex patches in domains with boundaries}

In this paper, we consider the patch evolution for the 2D Euler equations and for the modified SQG equations in
the presence of boundaries.  The latter are important in many
applications, in particular, in the onset of turbulence and in the creation of small scales in the motion of fluids.
The global existence of a single  $C^{1,\gamma}$ patch for the 2D Euler equations
on the half-plane $D:=\Rm\times\Rm^+$ was proved by Depauw \cite{d} when the patch does not touch the boundary
$\partial D$ initially. If it is, then \cite{d} only proved that the patch will remain $C^{1,\gamma}$ for a finite time, while Dutrifoy \cite{d2} proved a result which can be used to obtain global existence in the weaker space $C^{1,s}$ for some $s\in(0,\gamma)$. Uniqueness of solutions in the 2D Euler case follows from the work of Yudovich \cite{Yudth}.

Since we are not aware of a global existence result without a loss of regularity for (either one or multiple)  2D Euler patches
on the half-plane which may touch its boundary,  we will provide a proof of the global existence for such $C^{1,\gamma}$ patches here.
This contrasts with our main goal, proving finite time singularity formation for the modified SQG patch evolution with~$\alpha>0$
in domains with a boundary.  These two results together will then also establish existence of a phase transition in the behavior of solutions at $\alpha=0$.
For the sake of minimizing the technicalities, we do not strive for the greatest generality,
and will consider $H^3$ patches (as in \cite{g,CCCGW,GS}) on the half-plane, with small enough
$\alpha>0$ (that is, slightly more singular than the 2D Euler case $\alpha=0$).
Our initial condition $\omega_0$
will be the difference of characteristic functions of two patches with smooth boundaries.  The patches will initially touch
the boundary of the half-plane and, as was explained above,  the loss of  $H^3$ regularity or  self-intersections of their boundaries, as well as touches of the two patches, will all constitute a singularity.

The possible importance of boundaries in the formation of singularities in fluids has been illustrated by recent numerical simulations of Luo and Hou \cite{LuoHou}, which suggested a new scenario for
singularity formation in the 3D Euler equations.  The flow in this scenario is axi-symmetric  on a cylinder and so, in a way, can be viewed as two-dimensional (see \cite{CHKLSY} for a more detailed discussion).
The rapid growth of the vorticity in these simulations happens on the boundary of the cylinder.
The geometry of the construction we carry out in this work bears some similarity to this scenario, as well as to the geometry of the Kiselev-\v Sver\' ak
example of a solution to the 2D Euler equations with  a
double exponential growth of its vorticity gradient.  In particular, in all three instances, a hyperbolic fixed point of the
flow located on the boundary is involved. However the construction itself and the methods we use are quite different from earlier works.

\subsubsection*{The main results}

Let us now turn to the specifics.
As we said above, we will only consider
modified SQG evolution for small enough $\alpha>0,$ specifically  $\alpha \in (0,\frac 1{24}).$ The constraint $\alpha<\frac 1{24}$ comes from the currently available local well-posedness results,
while the singularity formation argument by itself  allows a somewhat larger value.
The Bio-Savart law for the patch evolution on the half-plane $D := \Rm\times\Rm^+$~is
\[
u=\nabla^\perp (-\Delta)^{-1+\alpha}\omega,
\]
with the Dirichlet Laplacian on $D$, which can also be written as
\begin{equation}
u(x, t) :=  \int_D \left( \frac{(x-y)^\perp}{|x-y|^{2+2\alpha}} -
\frac{(x-\bar y)^\perp}{|x-\bar y|^{2+2\alpha}} \right) \omega(y,t) dy
\label{eq:velocity_law}
\end{equation}
for $x \in \bar D$ (up to a positive pre-factor, which can be dropped without loss due to scaling). We use here the notation
\[
v^{\perp}:=(v_2, -v_1) \qquad\text{and}\qquad \bar v:=(v_1, -v_2)
\]
for $v=(v_1,v_2)$.  The vector field $u$ given by (\ref{eq:velocity_law})
is divergence free and tangential to the boundary~$\partial D$, that is,
\[
\hbox{$u_2(x,t)=0$ when $x_2=0$.}
\]
A traditional approach to the 2D Euler ($\alpha=0$) vortex patch evolution, going back to Yudovich (see~\cite{MP}
for an exposition) is via the corresponding  flow map. The active scalar $\omega$ is advected by~$u$ from (\ref{eq:velocity_law}) via
\begin{equation}\label{1.31}
\omega(x,t) = \omega \left(\Phi^{-1}_t(x),0\right),
\end{equation}
where
\begin{equation}\label{eq:alpha}
\frac{d}{dt}\Phi_t(x) = u\left(\Phi_t(x),t \right) \qquad\text{and}\qquad  \Phi_0(x)=x.
\end{equation}
The initial condition $\omega_0$ for  \eqref{eq:velocity_law}-\eqref{eq:alpha} is patch-like,
\begin{equation}\label{patchlikeid}
\omega_0 = \sum_{k=1}^N \theta_k \chi_{\Omega_{0k}},
\end{equation}
with $\theta_1,\dots,\theta_N\neq 0$ and $\Omega_{01},\dots,\Omega_{0N}\subseteq D$  bounded open sets, whose closures $\overline{\Omega_{0,k}}$ are pairwise disjoint and whose boundaries~$\partial\Omega_{0k}$ are  simple closed curves.

One reason the Yudovich theory works for the 2D Euler equations is that for $\omega$ which is (uniformly in time) in $L^1\cap L^\infty$, the velocity field
$u$ given by (\ref{eq:velocity_law}) with $\alpha=0$ is log-Lipschitz in space, and the flow map $\Phi_t$ is  everywhere well-defined. In our situation,
when $\omega$ is a patch solution and $\alpha>0$, the flow $u$ from (\ref{eq:velocity_law}) is smooth away from the patch boundaries  $\partial\Omega_k(t)$ but is only
H\"older at $\partial\Omega_k(t)$ which is exactly where one needs to use the flow map (see Lemma~\ref{lemma:uniform_u_bound}
for the corresponding H\"older estimate). Thus, the Yudovich definition of the evolution may not be applied directly, as the flow
trajectories need not be unique when $u$ is only H\"older continuous.
We will instead use a natural alternative definition of patch solutions to~(\ref{sqg})-(\ref{eq:velocity_law}), which will be equivalent to the
usual definition in the 2D Euler case, and closely related to the definitions used in earlier works on modified SQG patches.
In order to not interrupt this introduction, we postpone the precise discussion of these points
to Section~\ref{sec:def-flow} --- see Definition~\ref{D.1.1} and the rest of that section.

The following local well-posedness result is proved in the companion paper~\cite{KYZ1}.

\commentout{

 We introduce here the following Definition \ref{D.1.1} which, as we discuss below, encompasses various previously used definitions.  We start with a definition of some  norms of boundaries of domains in $\Rm^2$, letting here $\mathbb T:=[-\pi,\pi]$ with $\pm\pi$ identified.

\begin{definition} \label{D.1.0}
Let  $\Omega\subseteq \mathbb{R}^2$ be a bounded open set whose boundary $\partial\Omega$ is a simple closed $C^{1}$ curve with arc-length $|\partial\Omega|$.  We call a {\it constant speed parametrization} of $\partial\Omega$ any counter-clockwise parametrization $z:\mathbb{T}\to \mathbb{R}^2$ of $\partial\Omega$ with $|z'|\equiv \frac {|\partial\Omega|}{2\pi}$ on $\mathbb T$ (all such $z$ are translations of each other),
and we define $\|\Omega\|_{C^{m,\gamma}}:=\|z\|_{C^{m,\gamma}}$ and $\|\Omega\|_{H^{m}}:=\|z\|_{H^{m}}$.
\end{definition}

{\it Remark.}
It is not difficult to see (using \cite[Lemma 3.4]{KRYZ1}), that an $\Omega$ as above satisfies $\|\Omega\|_{C^{m,\gamma}}<\infty$ (resp.~$\|\Omega\|_{H^{m}}<\infty$) precisely when for some $r>0$, $M<\infty$, and each $x\in\partial\Omega$, the set $\partial\Omega\cap B(x,r)$ is (in the coordinate system centered at $x$ and with axes given by the tangent and normal vectors to $\partial\Omega$ at $x$) the graph of a function with $C^{m,\gamma}$ (resp. $H^m$) norm less than $M$.

Next, let $d_H(\Gamma,\tilde\Gamma)$ be the Hausdorff distance of sets $\Gamma,\tilde\Gamma$, and for a set $\Gamma\subseteq\Rm^2$, vector field $v:\Gamma\to\Rm^2$, and $h\in\Rm$, we let
\[
X_{v}^h[\Gamma]:= \{ x+hv(x)\,:\, x\in\Gamma\}.
\]
Our definition of patch solutions to \eqref{sqg}-\eqref{eq:velocity_law} on the half-plane is now as follows.

\begin{definition}\label{D.1.1}
Let $D:=\Rm\times\Rm^+$, let $\theta_1,\dots,\theta_N\in\Rm\setminus\{0\}$, and for each $t\in[0,T)$, let $\Omega_1(t),\dots,\Omega_N(t)\subseteq D$ be bounded open sets with pairwise disjoint closures whose boundaries $\partial \Omega_k(t)$ are simple closed curves, each  continuous in $t\in[0,T)$ with respect to Hausdorff distance of sets.  Let $\omega(\cdot,t) := \sum_{k=1}^N \theta_k \chi_{\Omega_k(t)}$ and denote $\Omega(t):=\bigcup_{k=1}^N \Omega_k(t)$.

If for  each $t\in(0,T)$ we have
\begin{equation}\label{1.3}
\lim_{h\to 0} \frac{d_H \Big(\partial\Omega(t+h),X_{u(\cdot,t)}^h[\partial\Omega(t)] \Big)}h = 0,
\end{equation}
with $u$ from \eqref{eq:velocity_law},
then $\omega$ is a patch solution to \eqref{sqg}-\eqref{eq:velocity_law} on the time interval $[0,T)$.  If we also have $\sup_{t\in [0,T']} \|\Omega_k(t)\|_{C^{m,\gamma}}<\infty$ (resp.~$\sup_{t\in [0,T']} \|\Omega_k(t)\|_{H^{m}}<\infty$)  for each $k$ and $T'\in(0,T)$,
then $\omega$ is a $C^{m,\gamma}$ (resp.~$H^m$) patch solution to \eqref{sqg}-\eqref{eq:velocity_law} on  $[0,T)$.
\end{definition}

{\it Remarks.}
1.   Continuity of $u$ (due to the last claim in Lemma \ref{lemma:uniform_u_bound} below) and \eqref{1.3} show that for patch solutions, $\partial\Omega$ is moving with velocity $u(x,t)$ at $t\in[0,T)$ and $x\in \partial\Omega(t)$.

2.  We note that our definition encompasses well-known definitions for the 2D Euler equation  in terms of \eqref{eq:alpha} and in terms of the normal velocity at $\partial\Omega$.  Indeed, if $\omega$ satisfies $\partial\Omega_{k}(t) = \Phi_t(\partial\Omega_{k}(0))$ for each $k$ and $t\in[0,T)$, the patches have pairwise disjoint closures, and their boundaries remain simple closed curves, then continuity of $u$, compactness of $\partial\Omega(t)$, and \eqref{eq:alpha} show that $\omega$ is a patch solution to  \eqref{sqg}-\eqref{eq:velocity_law} on  $[0,T)$.  Moreover, if $\partial\Omega(t)$ is $C^1$  and $n_{x,t}$ is the outer unit normal vector at $x\in \partial\Omega(t)$, then \eqref{1.3} is equivalent to motion of $\partial\Omega(t)$ with outer normal velocity $u(x,t)\cdot n_{x,t}$ at each $x\in\partial\Omega(t)$ (which can be defined, e.g., by \eqref{1.3} with $u(\cdot,t)$ replaced by $(u(\cdot,t)\cdot n_{\cdot,t})n_{\cdot,t}$).
However, Definition~\ref{D.1.1} can be stated even if $\Phi_t(x)$ cannot be uniquely defined for some $x\in\partial\Omega(0)$ (when $\alpha>0$, this could even be the case for $x\notin\partial\Omega(0)$, as the hypotheses of Theorem~\ref{T.1.1}(b) below suggest) or when $\partial\Omega(t)$ is not $C^1$.

3. It is not difficult to show (see \cite{KRYZ1}) that $C^1$ patch solutions to \eqref{sqg}-\eqref{eq:velocity_law} are also weak solutions to \eqref{sqg} in the sense that for each $f\in C^1(\bar D)$ we have
\begin{equation} \label{1.6}
\frac d{dt} \int_D \omega(x,t)f(x)dx = \int_D \omega(x,t) [u(x,t)\cdot\nabla f(x)] dx
\end{equation}
for all $t\in(0,T)$, with  both sides continuous in $t$.
Also, weak solutions to \eqref{sqg}-\eqref{eq:velocity_law} which are of the form $\omega(\cdot,t) := \sum_{k=1}^N \theta_k \chi_{\Omega_k(t)}$, with $C^1$ boundaries $\partial\Omega_k(t)$ which move with some continuous velocity $v:\Rm^2\times(0,T)\to\Rm^2$ (in the sense of \eqref{1.3} with $v$ in place of $u$), do satisfy \eqref{1.3} with $u$ (hence they are patch solutions if those boundaries are simple closed curves and the domains have pairwise disjoint closures).
Moreover, $|\Omega_k(t)|=|\Omega_k(0)|$ holds for each $k$ and $t\in[0,T)$.

 In the 2D Euler case $\alpha=0$, it is not difficult to show using standard results  that there is a unique global weak solution $\omega$ to \eqref{sqg} on $D$ with a given $\omega(\cdot,0)$ as in Definition \ref{D.1.1}, and it is of the form $\omega(\cdot, t) = \sum_{k=1}^N \theta_k \chi_{\Omega_{k}(t)}$, with $\partial\Omega_{k}(t) = \Phi_t(\partial\Omega_{k}(0))$ (we spell this argument out in Section~\ref{sec:euler}).  Remark 2 then shows that as long as the patch boundaries remain pairwise disjoint simple closed curves,
 $\omega$ is also a patch solution to  \eqref{sqg}-\eqref{eq:velocity_law}.  As was shown in  \cite{d2}, this is the case for all time when the initial data is one $C^{1,\gamma}$ patch (and we prove here that $\omega$ is also a $C^{1,\gamma}$ patch solution, even in the case of several patches).



On the other hand, our companion paper \cite{KRYZ1} proves local existence and uniqueness of $H^3$ patch solutions to \eqref{sqg}-\eqref{eq:velocity_law} on the half-plane for  small $\alpha>0$
(as well as on $\Rm^2$ for $\alpha\in(0,\frac 12)$).  Recall that uniqueness for patch solutions with $\alpha>0$ was previously only proved within a special class of SQG patches on $\Rm^2$ with $C^\infty$ boundaries in \cite{Rodrigo}.
In \cite{KRYZ1} we also address the relationship of  \eqref{1.3} to the maps  from \eqref{eq:alpha}. Note that since $u$ is obviously smooth away from $\partial\Omega$, $\Phi_t(x)$ remains unique at least until it hits $\partial\Omega$ (in the Euler case, $\Phi_t(x)$ is always unique because $u$ is log-Lipschitz), after which it still exists but need not be unique.   The results from \cite{KRYZ1}  imply that for $\alpha<\frac 14$ and patch solutions with $H^3$ boundaries, this remains true for any $x\in \bar D\setminus\partial\Omega(0)$ as long as the solution remains regular (this will prove useful in the proof of our main blow-up result below).  We collect the relevant results from \cite{KRYZ1} in the following theorem.

}

 \begin{theorem}\label{T.1.1}
(\cite{KYZ1})
If  $\alpha\in(0,\frac 1{24})$, then for each $H^3$ patch-like initial data $\omega_0$,  there exists a unique local $H^3$
patch solution $\omega$ to \eqref{sqg}-\eqref{eq:velocity_law} with $\omega(\cdot,0)=\omega_0$.  Moreover, if the maximal time
$T_\omega$ of existence of $\omega$ is finite, then at $T_\omega$ a singularity forms: either two patches touch, or a patch boundary touches
itself or  loses $H^3$ regularity.
\end{theorem}

The hypothesis $\alpha<\frac 1{24}$ in Theorem~\ref{T.1.1} may well be an artifact of the local existence proof, but we still will need a "small $\alpha$"
assumption, even though less restrictive, in the finite time singularity proof below.
The last claim in this theorem means that  either
\[
\hbox{$\partial\Omega_k(T_\omega)\cap \partial\Omega_i(T_\omega)\neq\emptyset$}
\]
 for some $k\neq i$, or $\partial\Omega_k(T_\omega)$ is not a simple closed curve for some $k$, or
\[
\lim_{t\nearrow T_\omega} \|\Omega_k(t)\|_{H^3}=\infty
\]
for some $k$, where the above norm is the $H^3$ norm of any constant-speed  parametrization of $\partial\Omega_k(t)$ (see Definition \ref{D.1.0} below).
Note that the sets
\[
\partial\Omega_k(T_\omega):=\lim_{t\nearrow T_\omega}\partial\Omega_k(t),
\]
 with the limit taken with respect to the Hausdorff distance $d_H$,  are well defined if $T_\omega<\infty$ because
$u$ is uniformly bounded --- see Lemma~\ref{lemma:uniform_u_bound} below.
In fact,  \cite[Lemma 4.10]{KYZ1} yields
\[
\hbox{$d_H(\partial\Omega(t),\partial\Omega(s))\le \|u\|_{L^\infty}|t-s|$}
\]
for $t,s\in[0,T_\omega)$.
%
%

We can now state the main results of the present paper --- global regularity of~$C^{1,\gamma}$ patch solutions
in the 2D Euler case $\alpha=0$, and existence of $H^3$ patch solutions which develop
a  singularity in finite time  for small $\alpha>0$.

\begin{theorem}\label{thmeuler1}
Let $\alpha =0$ and $\gamma\in(0,1]$.  Then for each $C^{1,\gamma}$ patch-like initial data
$\omega_0$,  there exists a unique global $C^{1,\gamma}$ patch solution $\omega$ to \eqref{sqg}-\eqref{eq:velocity_law}
with $\omega(\cdot,0)=\omega_0$.
\end{theorem}

\begin{theorem}\label{main1234}
Let $\alpha\in(0,\frac 1{24})$.  Then there are $H^3$ patch-like initial data $\omega_0$ for which the unique local $H^3$
patch solution $\omega$ to \eqref{sqg}-\eqref{eq:velocity_law} with $\omega(\cdot,0)=\omega_0$
becomes singular in finite time (i.e., its maximal time of existence $T_\omega$ is finite).
\end{theorem}


%

To the best of our knowledge, Theorem \ref{main1234} is the first rigorous proof of finite time singularity formation in this class of
fluid dynamics models.
 Moreover, Theorems~\ref{thmeuler1} and \ref{main1234} show that the $\alpha$-patch model
undergoes a phase transition at $\alpha=0$, which provides a reason for calling the~2D Euler equations ``critical''.


Let us now describe the type initial conditions, depicted in
Figure~\ref{fig_init}, which will lead to a singularity for $\alpha>0$.
\begin{figure}[htbp]
\begin{center}
\includegraphics[scale=1.5]{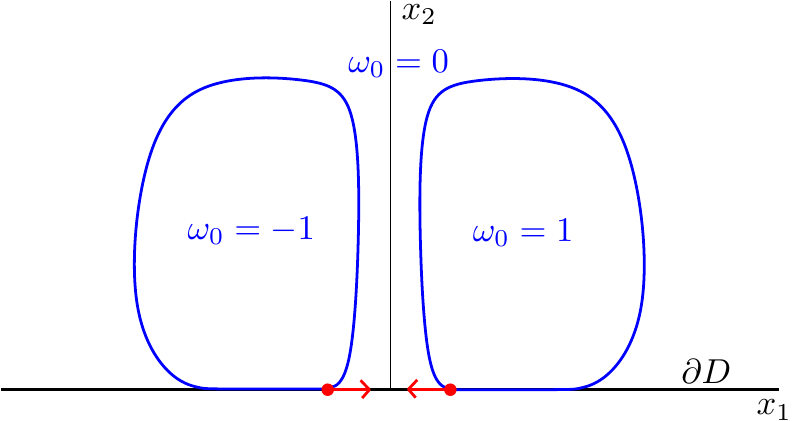}
\caption{Initial data $\omega_0$ which leads to a finite time singularity. \label{fig_init}}
\end{center}
\end{figure}
As we  have mentioned above, our choice of initial data is motivated by the numerical
simulations of the three-dimensional Euler equations in~\cite{LuoHou}, as well as~by the example of
smooth solutions to the 2D Euler equations with a double exponential temporal growth of their vorticity gradients in \cite{KS}.
The initial condition consist of two patches with opposite signs,
symmetric with respect to the $x_2$-axis and touching the $x_1$-axis.  The patches are sufficiently close to the origin and
have a sufficiently large area.  It can then be seen from \eqref{eq:velocity_law}
 that the rightmost point of the left patch on the $x_1$-axis and
the leftmost point of the right patch on the $x_1$-axis will move toward each other (see Figure~\ref{fig_init}).
In the case of the 2D Euler equations $\alpha=0$, Theorem \ref{thmeuler1} shows that the two points never reach the origin.
When $\alpha>0$ is small, however, we are able to control the evolution sufficiently well to show that --- unless the solution develops another
singularity earlier --- both points will reach the origin in a finite time. The argument yielding such control is fairly subtle, and
the estimates do not extend to all $\alpha<\frac 12$, even though one would expect the singularity formation to persist for more singular equations.

We note that we will actually run the singularity formation argument for the less regular $C^{1\gamma}$ patch solutions. However, we do not currently have
local well-posedness theorem in this class for $\alpha>0$, even though existence of such solutions follows from existence of the more regular
$H^3$ patch solutions. Since our argument requires odd symmetry, which would follow from uniqueness due to the symmetries of the equation, it effectively shows that there exist $C^{1,\gamma}$ patch solutions which
either have a finite maximal time of existence (i.e., exhibit singularity formation) or lose uniqueness (and odd symmetry).

Throughout the paper we denote by $C$, $C_\gamma$, etc.~various universal constants, which may change from line to line.

{\bf Acknowledgment.} We thank Peter Guba, Bob Hardt, and Giovanni Russo for useful discussions.
We acknowledge partial support by NSF-DMS grants 1056327, 1159133, 1311903,
1411857, 1412023, and 1535653.

\section{Vortex patches and low regularity velocity fields}\label{sec:def-flow}

In this section, we make precise the notion of the patch evolution for $\alpha>0$ and recall additional existence
results from~\cite{KYZ1} which we will need in the proof of Theorem \ref{main1234}.

\subsubsection*{The definition of the patch evolution}

As we mentioned above, H\"older regularity of
the fluid velocity $u$ at the patch boundaries is not sufficient for a unique definition of
the trajectories from \eqref{eq:alpha} when $\alpha>0$.
We start with a definition of the H\"older and Sobolev norms of the boundaries of domains in~$\Rm^2$ which will make
the notions of $C^{1,\gamma}$ and $H^3$ patches precise.

\begin{definition} \label{D.1.0}
Let  $\Omega\subseteq \mathbb{R}^2$ be a bounded open set whose boundary $\partial\Omega$
is a simple closed $C^{1}$ curve with arc-length $|\partial\Omega|$.  A constant speed parametrization of $\partial\Omega$
is any counter-clockwise parametrization $z:\mathbb{T}\to \mathbb{R}^2$ of $\partial\Omega$ with
$|z'|\equiv \tfrac 1{2\pi}{|\partial\Omega|}$ on the circle $\mathbb T:=[-\pi,\pi]$ (with $\pm\pi$ identified),
and we define $\|\Omega\|_{C^{m,\gamma}}:=\|z\|_{C^{m,\gamma}}$ and $\|\Omega\|_{H^{m}}:=\|z\|_{H^{m}}$.
\end{definition}
It is not difficult to see (using \cite[Lemma 3.4]{KYZ1}), that a domain $\Omega$ as above
satisfies~$\|\Omega\|_{C^{m,\gamma}}<\infty$ (resp.~$\|\Omega\|_{H^{m}}<\infty$) precisely when
for some $r>0$, $M<\infty$, and each~$x\in\partial\Omega$, the set $\partial\Omega\cap B(x,r)$ is
(in the coordinate system centered at $x$ and with the axes given by the tangent and normal vectors
to $\partial\Omega$ at $x$) the graph of a function with $C^{m,\gamma}$ (resp. $H^m$) norm less than $M$.

We denote by $d_H(\Gamma,\tilde\Gamma)$ the Hausdorff distance between two
sets $\Gamma,\tilde\Gamma$. For a set~$\Gamma\subseteq\Rm^2$, a vector field $v:\Gamma\to\Rm^2$, and $h\in\Rm$, we let
\[
X_{v}^h[\Gamma]:= \{ x+hv(x)\,:\, x\in\Gamma\}.
\]
Our definition of a patch solution to \eqref{sqg}-\eqref{eq:velocity_law} in the half-plane is as follows.
\begin{definition}\label{D.1.1}
Let $D:=\Rm\times\Rm^+$, let $\theta_1,\dots,\theta_N\in\Rm\setminus\{0\}$, and for each $t\in[0,T)$,
let~$\Omega_1(t),\dots,\Omega_N(t)\subseteq D$ be bounded open sets with pairwise disjoint closures whose
boundaries $\partial \Omega_k(t)$ are simple closed curves, such that each $\partial \Omega_k(t)$ is also continuous in $t\in[0,T)$ with respect to $d_H$.
Denote $\Omega(t):=\bigcup_{k=1}^N \Omega_k(t)$ and let
\begin{equation} \label{1.55}
\omega(x,t) := \sum_{k=1}^N \theta_k \chi_{\Omega_k(t)}(x).
\end{equation}

If for  each $t\in(0,T)$ and $u$ from \eqref{eq:velocity_law}, we have
\begin{equation}\label{1.3}
\lim_{h\to 0} \frac{d_H \Big(\partial\Omega(t+h),X_{u(\cdot,t)}^h[\partial\Omega(t)] \Big)}h = 0,
\end{equation}
then $\omega$ is a patch solution to \eqref{sqg}-\eqref{eq:velocity_law}  on the time interval $[0,T)$.
If we also have
\[
\sup_{t\in [0,T']} \|\Omega_k(t)\|_{C^{m,\gamma}}<\infty \qquad \left(\text{resp. } 
\sup_{t\in [0,T']} \|\Omega_k(t)\|_{H^{m}}<\infty \right)
\]
for each $k$ and $T'\in(0,T)$,
then $\omega$ is a $C^{m,\gamma}$ (resp.~$H^m$) patch solution to \eqref{sqg}-\eqref{eq:velocity_law} on~$[0,T)$.
\end{definition}
Lemma \ref{lemma:uniform_u_bound} below shows that $u$ is H\"older continuous for patch solutions, thus \eqref{1.3} says that
$\partial\Omega$ is moving with velocity $u(x,t)$ at any $t\in[0,T)$ and $x\in \partial\Omega(t)$.

This definition generalizes the well-known definitions for the 2D Euler equations in terms of \eqref{1.31}-\eqref{eq:alpha} or in terms of
the normal velocity at $\partial\Omega$.  Indeed, if $\omega$ satisfies
$\partial\Omega_{k}(t) = \Phi_t(\partial\Omega_{k}(0))$ for each $k$ and $t\in[0,T)$, the patches have pairwise disjoint closures,
and their boundaries remain simple closed curves, then continuity of $u$, compactness of $\partial\Omega(t)$, and \eqref{eq:alpha}
show that $\omega$ is a patch solution to~\eqref{sqg}-\eqref{eq:velocity_law}
on  $[0,T)$.  Moreover, if $\partial\Omega(t)$ is $C^1$  and $n_{x,t}$ is the outer unit normal vector at $x\in \partial\Omega(t)$,
then \eqref{1.3} is equivalent to the motion of $\partial\Omega(t)$ with the outer normal velocity $u(x,t)\cdot n_{x,t}$
at each $x\in\partial\Omega(t)$
(which can be defined in a natural way by \eqref{1.3} with $u(\cdot,t)$ replaced by $(u(\cdot,t)\cdot n_{\cdot,t})n_{\cdot,t}$).
However, Definition~\ref{D.1.1} makes sense even if $\Phi_t(x)$ cannot be uniquely defined for some
$x$, or when $\partial\Omega(t)$ is not~$C^1$.

It is not difficult to show~(see \cite{KYZ1}, Remark 3 after Definition 1.2) that $C^1$ patch solutions to \eqref{sqg}-\eqref{eq:velocity_law}
are also weak solutions to \eqref{sqg} in the sense that for each $f\in C^1(\bar D)$ we have
\begin{equation} \label{1.6}
\frac d{dt} \int_D \omega(x,t)f(x)dx = \int_D \omega(x,t) [u(x,t)\cdot\nabla f(x)] dx
\end{equation}
for all $t\in(0,T)$, with  both sides continuous in $t$.
Also, weak solutions to \eqref{sqg}-\eqref{eq:velocity_law} which are of the form \eqref{1.55}
and have $C^1$ boundaries $\partial\Omega_k(t)$ which move with some continuous velocity $v:\Rm^2\times(0,T)\to\Rm^2$
(in the sense of \eqref{1.3} with $v$ in place of $u$), do satisfy \eqref{1.3} with $u$ (hence they are patch solutions if those
boundaries are simple closed curves and the domains have pairwise disjoint closures).
Moreover, \eqref{1.6} also leads to $|\Omega_k(t)|=|\Omega_k(0)|$ for each $k$ and $t\in[0,T)$ --- see an elementary argument at the end
of the introduction of \cite{KYZ1}.


We also note that in the 2D Euler case $\alpha=0$, it is not difficult to show via the standard approach of Yudovich theory that there is a unique global weak
solution $\omega$ to \eqref{sqg}-\eqref{eq:velocity_law} on $D$ with a given $\omega(\cdot,0)$ as in Definition \ref{D.1.1}, and it is of
the form \eqref{1.55}
with $\partial\Omega_{k}(t) = \Phi_t(\partial\Omega_{k}(0))$. (We spell out this argument  in Section~\ref{sec:euler}.)
Thus, the above shows that as long as the patch boundaries remain pairwise disjoint simple closed curves,
$\omega$ is also the unique patch solution to  \eqref{sqg}-\eqref{eq:velocity_law}.

\subsubsection*{Relation of patch solutions to the flow map $\Phi_t$ in the modified SQG case $\alpha>0$}

The companion paper \cite{KYZ1}, which proves Theorem \ref{T.1.1}
as well as the same result on $\Rm^2$ for all $\alpha\in(0,\frac 12)$ (thus extending the results of~\cite{Rodrigo}
for infinitely smooth SQG patches of a special type on $\Rm^2$ to all $H^3$ modified SQG patches), also provides a link
between patch solutions and the flow map $\Phi_t$ from \eqref{eq:alpha} which will be important in our finite time singularity proof.
Note that since~$u$ is smooth away from $\partial\Omega$, the trajectories $\Phi_t(x)$ remain unique at least until they
hit~$\partial\Omega$  (in the Euler case, $\Phi_t(x)$ is always unique because $u$ is log-Lipschitz).
However, after hitting a patch boundary, the trajectory still exists but need not be unique.   Part (a) of the following result from \cite{KYZ1}  shows
that for $\alpha<\frac 14$ and patch solutions with $H^3$ boundaries, the flow lines which start away from $\partial\Omega(0)$
will stay away from $\partial\Omega(t)$
as long as the solution remains regular (note that we have $H^3(\mathbb T)\subseteq C^{1,1}(\mathbb T)$).
%

\begin{theorem}\label{T.1.1bis}
(\cite{KYZ1})
For $\omega$ as in the first paragraph of Definition~\ref{D.1.1} and $x\in\bar D \setminus \partial \Omega(0)$,  let $t_{\omega,x}\in [0,T]$ be the maximal time
such that the solution of \eqref{eq:alpha} with $u$ from \eqref{eq:velocity_law} satisfies $\Phi_t(x)\in \bar D\setminus \partial\Omega(t)$ for each
$t\in[0,t_{\omega,x})$.
%

(a) If $\alpha\in(0,\frac 14)$, $\gamma\in(\frac{2\alpha}{1-2\alpha},1]$, and $\omega$ is a $C^{1,\gamma}$ patch solution
to \eqref{sqg}-\eqref{eq:velocity_law} on  $[0,T)$, then $t_{\omega,x}= T$ for each  $x\in \bar D\setminus\partial\Omega(0)$ and
\[
\Phi_t:[\bar D\setminus \partial\Omega(0)]\to [\bar D\setminus \partial\Omega(t)]
\]
is a bijection for each $t\in[0,T)$.
\\

(b) If $\alpha\in(0,\frac 12)$, 
$t_{\omega,x}=  T$ for each  $x\in \bar D\setminus\partial\Omega(0)$, and
$\Phi_t:[\bar D\setminus \partial\Omega(0)]\to [\bar D\setminus \partial\Omega(t)]$ is a bijection for each $t\in[0,T)$,
then $\omega$ is a patch solution to \eqref{sqg}-\eqref{eq:velocity_law} on  $[0,T)$.  Moreover,  $\Phi_t$ is measure preserving on $\bar D\setminus \partial\Omega(0)$ and it also maps each $\Omega_k(0)$ onto $\Omega_k(t)$ as well as $\bar D\setminus \overline{\Omega(0)}$ onto $\bar D\setminus \overline{\Omega(t)}$. Finally,  we have
\[
\Phi_t(\partial\Omega_k(0)) = \partial\Omega_k(t)
\]
for each $k$ and $t\in[0,T)$, in the sense that any solution of \eqref{eq:alpha} with $x\in\partial\Omega_k(0)$
satisfies~$\Phi_t(x)\in \partial\Omega_k(t)$, and for each  $y\in\partial\Omega_k(t)$, there is $x\in\partial\Omega_k(0)$
and a solution of~\eqref{eq:alpha} such that $\Phi_t(x)=y$.
\end{theorem}

\section{Global well-posedness for the Euler case $\alpha=0$}\label{sec:euler}

In this section we prove Theorem \ref{thmeuler1}.

\subsection{Proof of Theorem \ref{thmeuler1} in the single patch case}
\label{sec:pf1}

For the sake of simplicity of
presentation, we first consider a single patch
$\Omega(t)\subseteq D$, with
\[
\omega(x,t)=\theta\chi_{\Omega(t)}(x).
\]
Later, we will show how to generalize this to finitely many patches.
We may assume without loss of generality that both $\theta=1$ and  $|\Omega(t)|=|\Omega(0)|=1$,
as the general single patch case then follows by a simple scaling.
The local-in-time existence and uniqueness of  $C^{1,\gamma}$ patch solutions for this initial value problem
was proved in \cite{d}. We will therefore focus on estimates which will allow the solution to be continued indefinitely.

Our approach is
a combination of the techniques form \cite{bc} and a refinement of the
estimates in \cite{d}.  Following \cite{bc}, we reformulate the vortex-patch evolution in
terms of the evolution of a function $\varphi(x,t)$, which defines the
patch via
\[
\Omega(t) = \{\varphi(x,t)>0\}.
\]
First, if $\partial\Omega(0)$ is a simple closed $C^{1,\gamma}$ curve,
there exists a function~$\varphi_0 \in C^{1,\gamma}(\overline{\Omega(0)})$ such that~$\varphi_0 >0$ on $\Omega(0)$, $\varphi_0 =0$ on $\partial\Omega(0)$, and
\begin{equation}\label{aug1902}
\inf_{\partial\Omega(0)}|\nabla \varphi_0| >0.
\end{equation}
One can obtain such $\varphi_0$, for instance, by solving the Dirichlet problem
\begin{align*}
-\Delta \varphi_0 &= f \hbox{ on $\Omega(0)$},\\
\varphi_0 &=0 \hbox{ on $\partial \Omega(0)$,}
\end{align*}
with $0\le f\in C^\infty_0(\Omega(0))$. It  follows from
the standard elliptic estimates  (see, e.g., \cite[Theorem 8.34]{GT}) that  $\varphi_0\in C^{1,\gamma}(\overline{\Omega(0)})$,
while (\ref{aug1902}) is a consequence of Hopf's lemma, which holds for $C^{1,\gamma}$ domains
by a result of Finn and Gilbarg \cite{FG} (see also \cite[Section 10]{HS}).

Consider now the flow map \eqref{eq:alpha}
corresponding to the Biot-Savart law for the Euler equation on the half plane,
\[
u(x,t) = \int_{\Omega(t)} \frac{(x-y)^{\perp}}{|x-y|^2} dy
- \int_{\tilde\Omega(t)}  \frac{(x-y)^{\perp}}{|x-y|^2} dy =: \V  (x,t) - \tv  (x,t),
\]
with $\tilde\Omega(t)$ the reflection of $\Omega(t)$ across the $x_1$-axis.
For~$x\in\Omega(t)$, we set
\[
\varphi(x,t)=\varphi_0(\Phi_t^{-1}(x)),
\]
 with $\Phi_t^{-1}$ the inverse map, so that $\varphi$ solves
\begin{equation}\label{2.3}
\partial_t\varphi+u\cdot\nabla\varphi=0
\end{equation}
on $\{(t,x)\,:\, t> 0 \text{ and } x\in\Omega(t)\}$.
Thus, for each $t\ge 0$, $\varphi(\cdot,t)>0$  on $\Omega(t)$, it vanishes on $\partial\Omega(t)$, and it is not defined on $\mathbb R^2\setminus \overline{\Omega(t)}$.
We now let
\begin{equation}\label{def:W}
\W = (w_1,w_2):= \nabla^\perp \varphi = (\partial_{x_2} \varphi, -\partial_{x_1} \varphi),
\end{equation}
and define
\begin{eqnarray*}
&&A_\gamma(t) := \|\W(\cdot,t)\|_{\dot C^\gamma(\Omega(t))} := \sup_{x,y\in\Omega(t)}\frac{|\W(x,t)-\W(y,t)|}{|x-y|^\gamma},
\\
&&
 A_\infty(t) := \|\W(\cdot,t)\|_{L^\infty(\Omega(t))},
\\
&&
A_{\inf}(t) := \inf_{x\in\partial \Omega(t)} |\W(x,t)|.
\end{eqnarray*}
By our choice of $\varphi_0$, we have
\[
A_\gamma(0),A_\infty(0), A_{\inf}(0)^{-1}<\infty.
\]
Moreover, $\W$ is divergence free and satisfies 
\begin{equation}
\W_t + u \cdot \nabla \W = (\nabla u) \W.
\label{eq_W}
\end{equation}
Proposition 1 in \cite{bc} and $|\Omega(t)|=1$ yield
\begin{equation} \label{2.1}
\|\nabla \V(\cdot,t)\|_{L^\infty(\mathbb{R}^2)} + \|\nabla \tv(\cdot,t)\|_{L^\infty(\mathbb{R}^2)} \leq C_\gamma  \left(1+\log_{+}\frac{A_\gamma(t)}{A_{\inf}(t)}\right),
\end{equation}
with $\log_+ x := \max\{\log x,0\}$ and some universal $C_\gamma<\infty$.
%
Hence, we obtain from \eqref{eq_W} (after doubling $C_\gamma$):
\begin{align}
A_{\infty}'(t) &\leq C_\gamma A_{\infty}(t)
\left(1+\log_+\frac{A_\gamma(t)}{A_{\inf}(t)}\right), \label{eq2} \\
A_{\inf}'(t) &\geq -  C_\gamma A_{\inf}(t)
\left(1+\log_+\frac{A_\gamma(t)}{A_{\inf}(t)}\right).  \label{eq1}
\end{align}

The main step in the proof will be to get an appropriate bound on $A_\gamma(t)$.  A simple calculation and \eqref{eq_W} yield
\begin{equation}\label{a_gamma_ineq}
A_\gamma'(t) \leq \gamma\|\nabla u(\cdot,t)\|_{L^\infty(\mathbb R^2)} A_\gamma(t) + \|\nabla u(\cdot,t) \W(\cdot,t)\|_{\dot C^\gamma(\Omega(t))}.
\end{equation}
Our goal will now be to show
 \begin{equation}
\|\nabla u(\cdot,t) \W(\cdot,t)\|_{\dot C^\gamma(\Omega(t))}
\leq C_\gamma A_\gamma(t) \left(1+\log_+ \frac{A_\gamma(t)}{A_{\inf}(t)}\right)
\label{eq:goal}
\end{equation}
with some universal $C_\gamma<\infty$.
This and \eqref{2.1} turn \eqref{a_gamma_ineq} into
\begin{equation}\label{123}
A_\gamma'(t) \leq C_\gamma  A_\gamma(t) \left(1+\log_+ \frac{A_\gamma(t)}{A_{\inf}(t)}\right).
\end{equation}
It follows from \eqref{123} and \eqref{eq1}
that the ratio
\[
A(t):=\frac{A_\gamma(t) }{A_{\inf}(t)}
\]
satisfies
\[
A'(t)\le C_\gamma A(t)(1+\log_+ A(t)).
\]
Therefore, $A(t)$ grows at most double-exponentially in time, and the same estimate
for~$A_\infty(t)$, $A_{\inf}(t)^{-1}$, and $A_\gamma(t)$ follows from \eqref{eq2}, (\ref{eq1}), and (\ref{123}), respectively.
 This then proves Theorem~\ref{thmeuler1}
for a single patch because  the above bounds on $A_\gamma(t)$, $A_{\inf}(t)$ and $A_\infty(t)$ imply that
the patch boundary cannot touch itself and  must be $C^{1,\gamma}$ at time $t$ (hence the local-in-time solution can be extended indefinitely).

Thus, the proof for a single patch is reduced to \eqref{eq:goal}.  The time variable will not play a role here,
so we will drop the argument $t$ in what follows. We split $(\nabla u) \W$ as
\[
(\nabla u)\W=(\nabla \V)  \W + (\nabla \tv)   \W.
\]
Since $\V $ is generated by the patch $\Omega$, and $\W$ is tangential to $\partial \Omega$, \cite[Corollary 1]{bc} gives
\begin{equation}\label{2.2}
\|(\nabla \V)  \W\|_{\dot C^\gamma(\Omega)} \le C_\gamma \|\nabla \V \|_{L^\infty(\mathbb R^2)}  \|\W\|_{\dot C^\gamma(\Omega)}
\end{equation}
with a universal $C_\gamma$.
Note that in \cite{bc}, $w$ is defined in $\Rm^2$ and all the norms are over $\mathbb R^2$.
We can use Whitney-type extension theorems \cite[Sec 6.2, Theorem 4]{stein} to extend
our $\varphi$ to all of $\mathbb R^2$ so that its $C^{1,\gamma}$ norm increases at most by a universal factor $\tilde C_\gamma<\infty$.
 This and \cite{bc} now yield \eqref{2.2}.
 Notice that this extended $\varphi$ does not necessarily solve \eqref{2.3}.

By \eqref{2.2} and \eqref{2.1},  $\|(\nabla \V ) \W\|_{\dot C^\gamma(\Omega)}$ is indeed bounded by the right-hand side of \eqref{eq:goal}.
Thus, it suffices to show that $\|(\nabla \tv)   \W\|_{\dot C^\gamma(\Omega)}$ satisfies the same estimate. As $\W$ is not tangential
to the boundary of $\tilde \Omega$, which generates $\tv  $,
we cannot directly apply the methods from \cite{bc}.
Let us take the above extension of $\varphi$ to $\mathbb R^2$ and define
\[
\tilde \varphi(x) := \varphi(\bar x) \qquad \text{and}\qquad
\tilde w:= -\nabla^\perp \tilde \varphi,
\]
with $\bar x=(x_1,-x_2)$.
Then $\tilde{\W}$ is tangential to $\partial \tilde\Omega$ and
\[
\| \tilde \W\|_{\dot C^\gamma(\mathbb R^2)}\le \tilde C_\gamma A_\gamma,
\]
Thus, Corollary 1 of~\cite{bc}  again yields
\[
\|(\nabla \tv )  \tilde{\W}\|_{\dot C^\gamma(\Omega)} \leq
C_\gamma \|\nabla \tv  \|_{L^\infty(\mathbb{R}^2)} \| \tilde \W\|_{\dot C^\gamma(\mathbb R^2)} \leq C_\gamma A_\gamma \left(1+\log_+\frac{A_\gamma}{A_{\inf}}\right).
\]
Hence, it suffices to prove  the following bound.
\begin{proposition}\label{proptech2de}
If $\varphi:\Omega\to[0,\infty)$ is positive on $\Omega\subseteq D$ and vanishes on $\partial\Omega$, then,
with~$\tv ,\W, \tilde \W, A_\gamma, A_{\inf}$  as above (and some universal $C_\gamma<\infty$) we have
\begin{equation}
\|\nabla \tv   (\W - \tilde{\W})\|_{\dot C^\gamma(\Omega)} \leq C_\gamma A_\gamma \left(1+\log_+\frac{A_\gamma}{A_{\inf}}\right).
\label{vww}
\end{equation}
\end{proposition}

Let us introduce some notation: for any $x\in\Rm^2\setminus\tilde\Omega$, define
\[
d(x) := \text{dist}(x,\tilde\Omega),
\]
 let $P_x\in\partial \tilde\Omega$
be such that ${\rm dist}(x,P_x)=d(x)$
(if there are multiple such points, we pick any of them), and let~$\bar{P}_x$ be the reflection of $P_x$ across the $x_1$-axis.
For an illustration of $\W, \tilde \W, d(x), P_x, \bar{P}_x$, see Figure \ref{w_w2}.

 \begin{figure}[htbp]
\begin{center}
\includegraphics[scale=0.7]{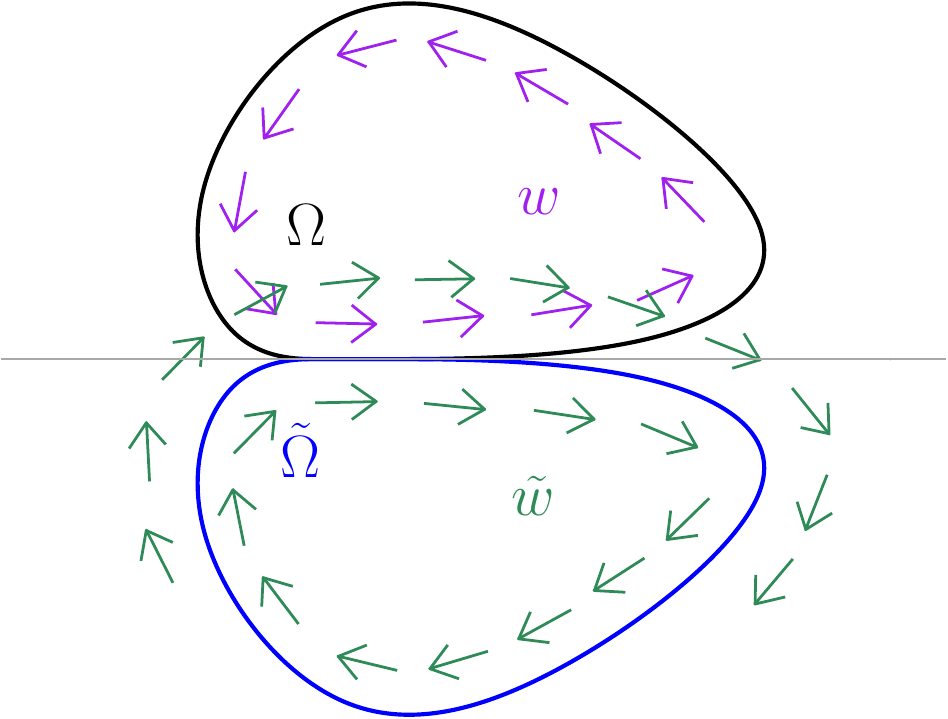}\quad\quad \includegraphics[scale=0.7]{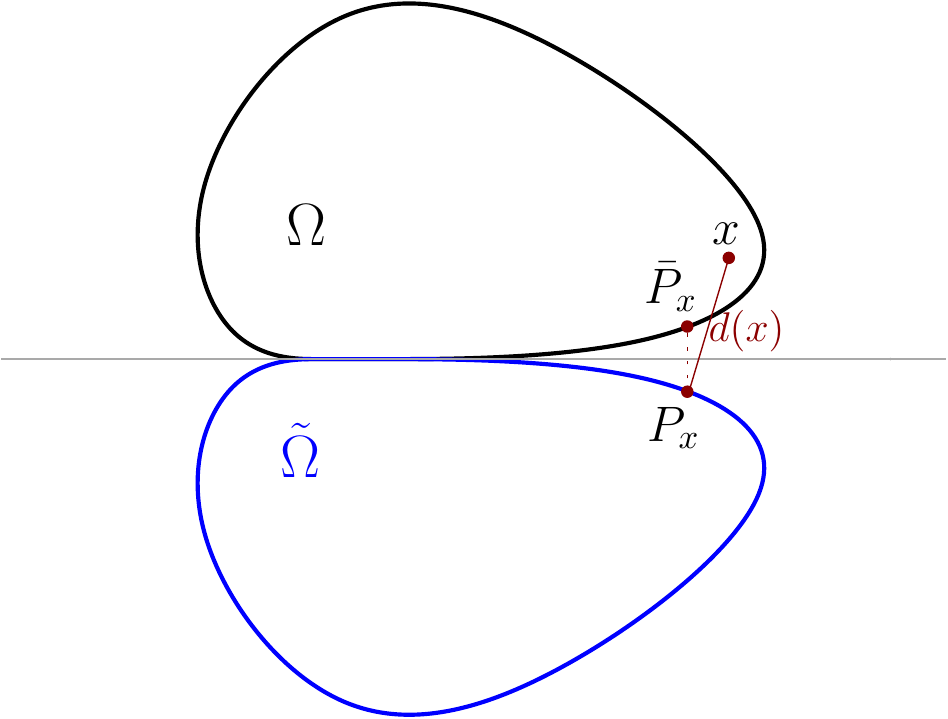}
\caption{Vector fields $\W$ and $\tilde{\W}$, and $d(x)$, $P_x$, $\bar{P}_x$ for $x\in \Omega$. \label{w_w2}}
\end{center}
\end{figure}

For arbitrary $x,y \in \Omega$ we can assume, without loss of generality, that $d(x)\leq d(y)$.  Then, with $g:= (\nabla \tv )  (\W - \tilde{\W})$,
we have
\[
\begin{split}
\frac{|g(x)-g(y)|}{|x-y|^\gamma} \leq & |\nabla \tv  (y)| \|\W - \tilde{\W}\|_{\dot C^\gamma(\Omega)}+ \underbrace{\frac{|\nabla \tv  (x) - \nabla \tv  (y)|}{|x-y|^\gamma}}_{T_1(x,y)} \underbrace{\big|\W(x) - \tilde{\W}(x)\big|}_{T_2(x)}.
\end{split}
\]
Since the first term on the right-hand side is bounded by the right-hand side of \eqref{vww}
due to \eqref{2.1} and the definition of $\tilde \W$, we only need to obtain the same bound for the second term.
We will estimate $T_1$ and $T_2$ separately, in terms of $A_\gamma$, $A_{\inf}$, $d(x)$, and $|\tilde{\W}(P_x)|=|{\W}(\bar P_x)|$.

Let us start with $T_2$.  We estimate
\[
T_2(x) \leq \big|\W(\bar P_x) - \tilde{\W}(P_x)\big| +\big|\W(\bar P_x) - \W(x)\big| + \big|\tilde{\W}(P_x) - \tilde{\W}(x)\big| \le 2\tilde C_\gamma A_\gamma d(x)^\gamma + 2|w_2(\bar P_x)|,
\]
where we used the inequality
\[
{\rm dist}(x, \bar P_x) \leq {\rm dist}(x,P_x)= d(x)
\]
 to bound the last two terms in the middle expression by $\tilde C_\gamma A_\gamma d(x)^\gamma$,
 while the first term equals $2|w_2(\bar P_x)|$ because
 \[
 \tilde \W(P_x)=(\W_1(\bar P_x),-\W_2(\bar P_x)).
 \]
 The following lemma will allow us to control $|w_2(\bar P_x)|$.

\begin{lemma}
For any $P = (p_1, p_2) \in \partial \Omega$, we have $|w_2(P)| \leq 2 \left(A_\gamma p_2^\gamma |\W(P)|^\gamma\right)^{\tfrac{1}{1+\gamma}}$.
\end{lemma}

\begin{proof}
Denote by $\theta \in[0, \frac \pi 2 ]$ the angle between $\nabla \varphi(P)$ and the $x_2$-axis
(see Figure \ref{theta_fig}), so that
\begin{equation}\label{aug1906}
|w_2(P)| = |\nabla \varphi(P)| \sin \theta\le 2 |\nabla \varphi(P)| \sin \frac\theta 2.
\end{equation}
If $\theta=0$, then we are done. Otherwise,
let $\nu$ denote the unit vector such that the angle between $\nu$ and $\nabla \varphi(P)$ is
$\frac{\pi}{2}-\frac{\theta}{2}$ 
(so $\nu$ points inside $\Omega$ at $P$) and $\nu_2<0$. Draw a ray in the direction $\nu$ and originating at $P$, and denote
by $Q$ its intersection with the $x_1$-axis. Note that $Q\neq P$ since $p_2>0$ due to $\theta\neq 0$.
\begin{figure}[htbp]
\begin{center}
\includegraphics[scale=0.9]{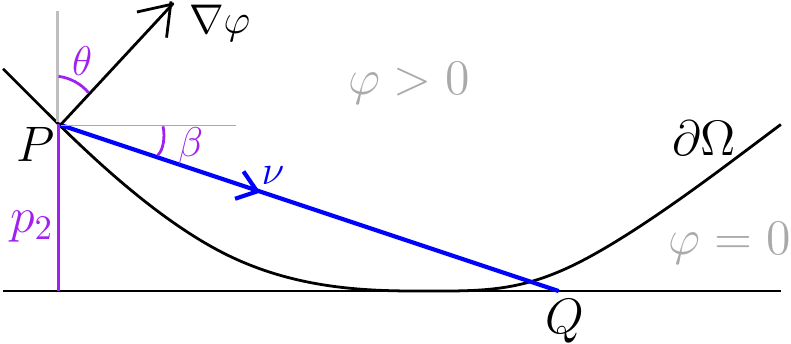}
\caption{The definitions of $\theta,\beta,\nu,Q$.\label{theta_fig}}
\end{center}
\end{figure}

The length of the segment $PQ$ is
\[
|PQ|= \frac{p_2}{\sin \beta},
\]
where either $\beta=\frac\theta 2$ or $\beta= \frac{3\theta} 2$, the latter if $(\nabla\varphi(P))_2<0$. In either case we have
\[
|PQ|\le\frac{p_2}{\sin(\theta/2)}.
\]
We also have
\[
\nabla\varphi(P) \cdot \nu = |\nabla \varphi(P)| \sin\frac{\theta}{2} > 0,
\]
and $\nabla\varphi \cdot \nu$  must change sign   on the segment $PQ$ because $Q\notin \Omega$ and $\varphi=0$
on $\partial\Omega$. As
\[
\|\nabla \varphi\|_{\dot C^\gamma(\Omega)} \leq A_\gamma,
\]
we obtain
\[
|\nabla \varphi(P)| \sin \frac{\theta}{2}
\leq A_\gamma \left(\frac{p_2}{\sin(\theta/ 2)} \right)^\gamma.
\]
Raising this to power $\frac 1{1+\gamma}$ and using (\ref{aug1906}) yields
\[
| w_2(P)|
\leq 2|\nabla \varphi(P)| \sin \frac\theta 2 \leq 2 \left(A_\gamma p_2^\gamma |\nabla \varphi(P)|^\gamma\right)^{\frac{1}{1+\gamma}}.
\]
Since $|\nabla \varphi(P)| = |\W(P)|$, the proof is complete.
\end{proof}

The above lemma applied at $P:=\bar P_x$, along with $|\W(\bar P_x)| =  |\tilde \W(P_x)|$,
now yields
\begin{equation}
T_2 (x)\leq 2\tilde C_\gamma A_\gamma d(x)^\gamma +  4\left(A_\gamma\, d(x)^\gamma |\tilde \W(P_x)|^\gamma\right)^{\frac{1}{1+\gamma}}.
\label{bound_t2}
\end{equation}
Next we bound $T_1$.
\begin{proposition}\label{T1prop}
With the hypotheses of Proposition \ref{proptech2de},  for $x,y\in\Omega$ with $d(x)\le d(y)$ we have
\begin{equation}
T_1(x,y) := \frac{|\nabla \tv  (x) - \nabla \tv  (y)|}{|x-y|^\gamma} \leq C_\gamma \left(1+\log_+\frac{A_\gamma}{A_{\inf}}\right) \min\left\{\frac{A_\gamma}{| \tilde \W(P_x)|}, d(x)^{-\gamma}\right\}.
\label{refined}
\end{equation}
\end{proposition}

A related but weaker bound (which does not suffice here) was proved in \cite{d}.
Before proving Proposition~\ref{T1prop}, let us first complete the proof of Proposition~\ref{proptech2de}.

\begin{proof}[Proof of Proposition~\ref{proptech2de}]
The bound \eqref{refined} implies
\[
T_1 (x,y) \leq C_\gamma  \left(1+\log_+\frac{A_\gamma}{A_{\inf}}\right) \min\left\{ \left( \frac{A_\gamma}{| \tilde \W(P_x)|} \right)^{\tfrac{\gamma}{1+\gamma}} d(x)^{-\tfrac{\gamma}{1+\gamma}}, d(x)^{-\gamma}\right\}.
\]
Multiplying this by \eqref{bound_t2}  gives
\[
T_1(x,y) T_2(x) \leq C_\gamma A_\gamma \left(1+\log_+\frac{A_\gamma}{A_{\inf}}\right).
\]
As we have explained above, this yields \eqref{vww} and   concludes the proof.
\end{proof}

We are left with proving Proposition~\ref{T1prop}.
We start with the following simple lemma.
\begin{lemma}\label{depauw1}
When $d(x)\le d(y)$ for $x,y\in\Omega$, we have (with a universal $C<\infty$)
\begin{equation}\label{depauweasy}
\frac{|\nabla \tv  (x) - \nabla \tv  (y)|}{|x-y|^\gamma} \leq \frac{C}\gamma d(x)^{-\gamma}.
\end{equation}
\end{lemma}

\begin{proof}
The mean value theorem yields
\[
\frac{|\nabla \tv  (x) - \nabla \tv  (y)|}{|x-y|^\gamma} \leq |\nabla^2 \tv(Z_{xy})| |x-y|^{1-\gamma}
\]
for some point $Z_{xy}$ on the segment connecting $x$ and $y$.
Since $\tilde\Omega$ is the reflection of~$\Omega\subseteq D$ with respect to
the $x_1$-axis, we have $d(x)\in[x_2, 2x_2]$ and $d(y)\in[y_2, 2y_2]$.
As $d(x)\le d(y)$, we then obtain
\[
d(Z_{xy}) \geq {\rm min}\{x_2, y_2\} \geq \frac {d(x)}2.
\]
Moreover, for any $Z\in \Rm^2\setminus\tilde\Omega$ we have (with a universal $C<\infty$)
\begin{equation}\label{hesscon}
|\nabla^2 \tv(Z)| \leq \int_{\Rm^2 \setminus B(Z,d(Z))} \frac{C}{|Z-z|^3}\,dz \leq Cd(Z)^{-1}.
\end{equation}
Combining these estimates, we obtain
\[
\frac{|\nabla \tv  (x) - \nabla \tv  (y)|}{|x-y|^\gamma} \leq Cd(Z_{xy})^{-1} |x-y|^{1-\gamma} \leq 2C d(x)^{-1}|x-y|^{1-\gamma}.
\]
If $|x-y| \leq d(x)$, then (\ref{depauweasy}) follows because $\gamma\le 1$.

If $|x-y| \geq d(x)$,  let
\[
Q_{xy}= (x_1, x_2+2|x-y|),
\]
and connect $x$ and $y$ by a path consisting of the segments $[xQ_{xy}]$ and $[Q_{xy}y]$.
Then
\begin{equation}\label{aug2002}
|Q_{xy}-y|\leq 3|x-y|
\end{equation}
yields
\begin{equation} \label{2.11}
 \frac{|\nabla \tv  (x) - \nabla \tv  (y)|}{3|x-y|} \leq  \int_0^1 |\nabla^2 \tv(x+s(Q_{xy}-x))|\,ds + \int_0^1 |\nabla^2 \tv(y+s(Q_{xy}-y))|\,ds.
  \end{equation}
Note that
\begin{equation}\label{aug2008}
d(x+s(Q_{xy}-x)) \geq \max\{ d(x), 2s|x-y|\},
\end{equation}
and we also have
\[
d(y+s(Q_{xy}-y)) \geq s|x-y| \ge \frac s3 |Q_{xy}-y|
\]
due to
\[
(Q_{xy}-y)_2 \ge |x-y|
\]
and  \eqref{aug2002}.  It then follows  that
\[
d(y)\le d(y+s(Q_{xy}-y))+s|Q_{xy}-y|\le  4d(y+s(Q_{xy}-y)),
\]
so by $d(x)\le d(y)$ and the above we have
\[
d(y+s(Q_{xy}-y)) \geq \max \left\{ \frac{d(x)}{4}, s|x-y| \right\},
\]
in addition to (\ref{aug2008}).
Combining these estimates with \eqref{hesscon}, we obtain
\begin{equation*}
\begin{split}
 \frac{|\nabla \tv  (x) - \nabla \tv  (y)|}{|x-y|^\gamma} &\leq
C|x-y|^{1-\gamma} \left(\int_0^{\frac{d(x)}{|x-y|}} d(x)^{-1} ds +  \int_{\frac{d(x)}{|x-y|}}^1 (s|x-y|)^{-1} ds\right) \\
  &\leq C |x-y|^{-\gamma} \left( 1+ \log \frac{|x-y|}{d(x)} \right)
\leq \frac{C}\gamma d(x)^{-\gamma}
\end{split}
\end{equation*}
because $|x-y|\ge d(x)$.
\end{proof}

We continue the proof of Proposition~\ref{T1prop}.
Due to Lemma~\ref{depauw1}, to prove \eqref{refined} we only need to consider
the case
\[
d(x) \leq \tilde  C^{-1} \left(\frac{| \tilde \W(P_x)|}{A_\gamma}\right)^{1/\gamma}
\]
for any fixed $\tilde C<\infty$.  Let us pick $\tilde C:= 16(4\tilde C_\gamma)^{1/\gamma}$, with the universal constant $\tilde C_\gamma$ from the remark about Whitney extensions after \eqref{2.2}, so that if we let $\tilde A_\gamma := \|\tilde\W\|_{\dot C^\gamma(\mathbb{R}^2)}$ and
\[
r_x:=\left(\frac{| \tilde \W(P_x)|}{2\tilde A_\gamma}\right)^{1/\gamma},
\]
it suffices to consider $d(x) \leq 2^{-4-1/\gamma}r_{x}$ (because $\tilde A_\gamma \leq \tilde C_\gamma A_\gamma$).

Hence, the next lemma finishes the proof of Proposition~\ref{T1prop}.



\begin{lemma}\label{T1key}
When $d(x) \leq \min\{d(y), 2^{-4-1/\gamma}r_{x}\}$ for $x,y\in\Omega$,
we have (with a universal constant~$C_\gamma<\infty$)
\begin{equation}
 \frac{|\nabla \tv  (x) - \nabla \tv  (y)|}{|x-y|^\gamma}
\leq C_\gamma \left(1+\log_+\frac{A_\gamma}{A_{\inf}}\right) \frac{A_\gamma}{| \tilde \W(P_x)|}.
\label{goal2}
\end{equation}
\end{lemma}

In the proof of this lemma, the following improvement of \eqref{hesscon} will
be used to control~$|\nabla \tv  (x) - \nabla \tv  (y)|$.
Its proof is postponed until the end of this section.
\begin{lemma}\label{lemma:d2vx}
For any $x\in \Rm^2\setminus\tilde\Omega$ with $d(x) \in(0,\frac 14r_{x}]$,
we have (with a universal $C_\gamma<\infty$)
\[
|\nabla^2 \tv  (x)| \leq C_\gamma d(x)^{-1+\gamma} r_{x}^{-\gamma}.
\]
\end{lemma}

\begin{proof}[Proof of Lemma \ref {T1key}]
Let us first assume
\[
|x-y| \geq 2^{-4-1/\gamma} r_{x},
\]
so that
\[
|x-y|^{-\gamma} \leq 64 \tilde C_\gamma\frac{A_\gamma}{|\tilde \W(P_x)|}.
\]
Then \eqref{goal2}  follows from
the estimate
\[
|\nabla \tv  (x)-\nabla \tv  (y)| \leq 2\|\nabla \tv  \|_{L^\infty(\Rm^2)}
\]
and  \eqref{2.1} (the latter holds for any $\Omega,\varphi$ as in Proposition \ref {proptech2de} --- see \cite[Proposition 1]{bc}).


Assume now that $|x-y|< 2^{-4-1/\gamma}  r_{x}$.
As in Lemma~\ref{depauw1}, let
\[
Q_{xy} = (x_1, x_2+2|x-y|),
\]
and connect the points $x$ and $y$ by a path consisting of the two segments $[xQ_{xy}],~[Q_{xy}y]$, again parametrized  by
\[
z_1(s) = x + s(Q_{xy}-x) \qquad\text{and}\qquad z_2(s) = y + s(Q_{xy}-y),
\]
for $s\in[0,1]$ (see Figure \ref{path}).
Then we again have
\[
d(z_i(s)) \geq s|x-y|
\]
 for $i=1,2$ and $s\in[0,1]$.
\begin{figure}[htbp]
\begin{center}
\includegraphics[scale=0.8]{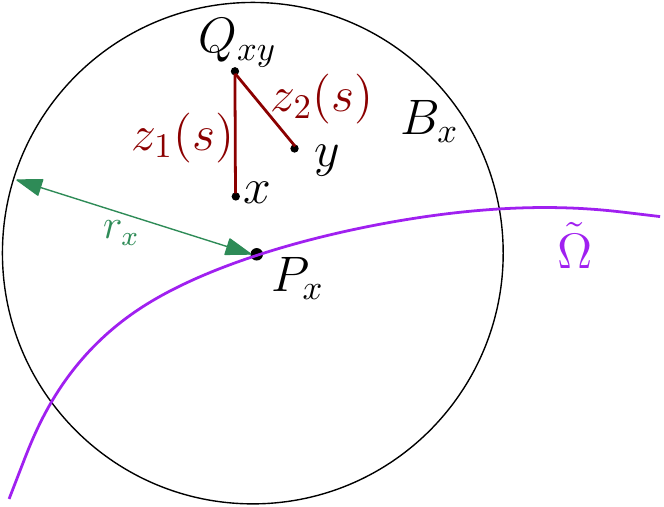}
\caption{The point $Q_{xy}$ and the paths $z_1(s)$ and $z_2(s)$.\label{path}}
\end{center}
\end{figure}

We also have
\[
|z_i(s)- P_x| \leq |z_i(s)-x| + d(x) \le 2|x-y|+d(x),
\]
so
\begin{equation}\label{aug2104}
d(z_i(s))\le 2^{-2-1/\gamma}r_{x}.
\end{equation}
These   imply
\begin{equation}\label{aug2102}
P_{z_i(s)} \in B(P_x, 2^{-1-1/\gamma}r_{x})\subseteq B_x:={B(P_x, r_{x})}.
\end{equation}
Note that for all $z\in B_x$ we have
\begin{equation}
|\tilde{\W}(z)-\tilde{\W}(P_x)|\leq \frac{|\tilde{\W}(P_x)|}{2}.
\label{eq:RPx}
\end{equation}
Thus, (\ref{aug2102})   gives
\[
|\tilde \W(P_{z_i(s)})| \geq \frac{1}{2} |\tilde \W(P_x)|,
\]
implying
\[
r_{{z_i(s)}} \geq 2^{-1/\gamma} r_{x}.
\]
From (\ref{aug2104}) it now follows that
\[
d(z_i(s))\le \tfrac 14 r_{z_i(s)}.
\]
Thus, Lemma \ref{lemma:d2vx} applies to  $z_i(s)$ and yields (together with the above estimates)
\[
 |\nabla^2 \tv  (z_i(s))| \leq C_\gamma d(z_i(s))^{-1+\gamma} r_{{z_i(s)}}^{-\gamma} \leq 2C_\gamma (s|x-y|)^{-1+\gamma} r_{{x}}^{-\gamma}.
\]
Then \eqref{2.11} implies
 \[
 \frac{|\nabla \tv  (x) - \nabla \tv  (y)|}{|x-y|^\gamma}
 \leq   12C_\gamma |x-y|^{1-\gamma} \int_0^1 \left(s|x-y| \right)^{-1+\gamma} r_{x}^{-\gamma}ds
 \leq \frac{12C_\gamma}\gamma r_{x}^{-\gamma} \le \frac{24C_\gamma\tilde C_\gamma}\gamma \frac{A_\gamma}{|\tilde \W(P_x)|},
\]
which gives \eqref{goal2}.
\end{proof}

\subsection{Proof of Theorem \ref{thmeuler1} in the general case}
\label{sec:pf2}


We now consider an initial condition $\omega_0$ with an arbitrary number of patches and arbitrary values of $\theta_k$
as in the statement of Theorem \ref{thmeuler1}, and extend it as an odd function  to~$x_2<0$.
By \cite[Theorems 8.1 and 8.2]{mb}, there is a unique global weak solution $\omega$ to \eqref{sqg} with the whole plane flow
\begin{equation}
u(x, t) =  \int_{\Rm^2}  \frac{(x-y)^\perp}{|x-y|^{2}} \omega(y,t) dy,
\label{eq:velocity_law_R}
\end{equation}
and the initial data $\omega(\cdot,0)=\omega_0$, in the sense that
\[
\int_D \omega (x,T)g(x,T)dx - \int_D \omega_0(x)g(x,0)dx = \int_{D\times(0,T)} \omega(x,t)[\partial_t g(x,t)+u(x,t)\cdot\nabla g(x,t)]dxdt
\]
for all $T<\infty$ and  $g\in C^1(\bar D\times [0,T])$.
This  solution is also a collection of vortex patches
\[
\omega(\cdot, t) = \sum_{k=1}^N \theta_k \chi_{\Omega_{k}(t)},
\]
with $\Omega_{k}(t) = \Phi_t(\Omega_{k}(0))$ for each $k$  \cite[Chapter 2, Theorem 3.1]{MP}.  Note that $\Phi_t(x)$ is uniquely defined for any $x\in\Rm^2$, due to the time-uniform log-Lipschitz apriori bound
\begin{equation} \label{2.23a}
 |u(x,t)-u(y,t)| \leq C_{\omega_0} |x-y| \log \left( 1 + |x-y|^{-1} \right)
\end{equation}
for $u$  (see, e.g., \cite[Lemma 8.1]{mb}), with the constant   depending only on  $\|\omega_0\|_{L^1}$ and~$\|\omega_0\|_{L^\infty}$.
Uniqueness shows that $\omega$ remains odd in $x_2$, thus its restriction to $D\times[0,\infty)$ is also the unique weak solution
to \eqref{sqg}, \eqref{eq:velocity_law_R}
(and it is unique such with $\omega(\cdot,0)=\omega_0$ because an odd-in-$x_2$ extension of a weak solution on $D\times[0,\infty)$ is a weak
solution on $\Rm^2\times[0,\infty)$).
It follows from \eqref{eq:alpha}, continuity of $u$ (which is obtained as the last claim in Lemma \ref{lemma:uniform_u_bound}
below but using  \eqref{2.23a} instead of \eqref{uHold}), and compactness of $\partial\Omega(t)\times\{t\}$ that \eqref{1.3}
holds for each $t>0$.  Hence, if we show that~$\{\partial\Omega_k(t)\}_{k=1}^N$ is a family of disjoint simple closed curves
for each $t\ge 0$, and
\[
\sup_{t\in[0,T]} \max_k \|\Omega_k(t)\|_{C^{1,\gamma}}<\infty
\]
for each $T<\infty$, then $\omega$ will also be a $C^{1,\gamma}$ patch solution to \eqref{sqg}-\eqref{eq:velocity_law} on $[0,\infty)$.  Moreover, since $C^{1,\gamma}$ patch solutions are weak solutions in the above sense as well (it is easy to see that \eqref{1.6} implies this), $\omega$ must then also be the unique patch solution.

Note that \eqref{2.23a} yields
\[
\min_{i\neq k} {\rm dist} ( \Omega_i(t),  \Omega_k(t)) \ge \delta(t)>0
\]
for all $t\ge 0$, where $\delta(t)$ decreases double exponentially in time.
This will ensure that the effects of the patches on each other will be controlled.  Therefore, it remains to prove that
each $\partial\Omega_k(t)$ is a simple closed curve with  $\|\partial\Omega_k(t)\|_{C^{1,\gamma}}$ uniformly bounded on bounded intervals.

Let us decompose
\[
u=\sum_{i=1}^N u_i,
\]
with each $u_i$ coming from the contribution of the patch $\Omega_i$ to $u$.  If $i\neq k$, then
obviously
\[
\|\nabla^n u_i(\cdot, t)\|_{L^\infty(\Omega_k(t))} \leq C(\omega_0,n) \delta(t)^{-n-1},
\]
for all $n\geq 0$.
This yields
\[
\|\nabla u_i(\cdot,t)\|_{\dot C^\gamma(\Omega_k(t))} \leq C(\omega_0) \delta(t)^{-3}
\]
for $i\neq k$.  Also, simple scaling shows that \eqref{2.1} now becomes (for each $i$ and with $\V_i,\tv_i$ defined analogously to $\V,\tv$)
\[
\|\nabla \V_i(\cdot,t)\|_{L^\infty(\mathbb{R}^2)} + \|\nabla \tv_i(\cdot,t)\|_{L^\infty(\mathbb{R}^2)} \leq C_\gamma |\theta_i| \left(1+\log_{+}\frac{A_\gamma(t)  |\Omega_i(t)|^{\gamma/2}}{A_{\inf}(t)}\right).
\]

We now consider a separate $\varphi_k$ and $w_k:=\nabla^\perp \varphi_k$ for each $\Omega_k$, all
$\varphi_k$ evolving with velocity $u$.  We also add $\sup_k$ in the definitions of $A_\gamma$ and $A_\infty$ and
$\inf_k$ in the definition of~$A_{\inf}$.  We can repeat the proof above, with
 \eqref{eq:goal} replaced by (for each $k$ and  $t>0$)
\begin{equation*}
\begin{split}
\|(\nabla u) \W_k\|_{\dot C^\gamma(\Omega_k)} &\leq C_\gamma \Theta A_\gamma
\left(|\Omega|+\log_+ \frac{A_\gamma}{A_{\inf}}\right) \\
&+ \sum_{i\neq k} \left(\|\nabla u_i\|_{L^\infty(\Omega_k)} \|\W_k\|_{\dot C^\gamma(\Omega_k)} + \|\nabla u_i\|_{\dot C^\gamma(\Omega_k)}
\|\W_k\|_{L^\infty(\Omega_k)}\right)\\
&\leq C_\gamma N \Theta A_\gamma
\left(|\Omega|+\log_+ \frac{A_\gamma}{A_{\inf}}\right)
+ C(\omega_0)N \delta^{-3} A_\infty,
\end{split}
\end{equation*}
where
\[
\Theta:=\max_{1\le k\le N}|\theta_k|
\qquad\text{and}\qquad |\Omega|:= 1+\max_{1\le k\le N}|\Omega_k(t)| = 1+\max_{1\le k\le N}|\Omega_k(0)|.
\]
Then \eqref{123} is replaced by
\begin{equation*}
A_\gamma'(t) \leq C_\gamma N \Theta A_\gamma (t)
\left(|\Omega|+\log_+ \frac{A_\gamma(t)}{A_{\inf}(t)}\right)
+ C(\omega_0)N \delta(t)^{-3} A_\infty (t).
\end{equation*}
From this and \eqref{eq2}, \eqref{eq1} a simple computation shows that
\[
\tilde A(t) := A_\gamma(t) A_{\inf}(t)^{-1} + A_\infty(t)
\]
satisfies
\begin{equation*}
\tilde A'(t) \leq C(\gamma, N, \omega_0) \tilde A(t) \left(\delta(t)^{-3} + \log_+\tilde A(t)\right).
\end{equation*}
Since $\delta(t)^{-3}$ increases at most double exponentially in time, it follows that $\tilde A(t)$ increases at most triple exponentially.
As before, this implies that each $\partial\Omega_k(t)$ is a simple closed curve with  $\|\partial\Omega_k(t)\|_{C^{1,\gamma}}$ uniformly bounded on bounded intervals.  Hence $\omega$ is a global~$C^{1,\gamma}$ patch solution
to  \eqref{sqg}-\eqref{eq:velocity_law}, thus  finishing the proof.

\subsection{Proof of Lemma \ref{lemma:d2vx}}
\label{sec:pfl}

Let us start with a simple geometric result concerning the behavior of $\partial \tilde\Omega$ near $P_x$, which is similar to the Geometric Lemma in \cite{bc}.  It says that $\partial \tilde\Omega\cap B_x$ is sufficiently ``flat''.
\begin{lemma} \label{geometry_lemma}
Given $x\in \Rm^2\setminus\tilde\Omega$, let $n_x:= \nabla \tilde \varphi(P_x) /|\nabla \tilde \varphi(P_x)|$, and
\begin{equation}
S_x:= \left\{P_x + \rho \nu \,:\, \rho\in[0,r_{x}), \, |\nu| = 1,\,  \left(\frac{\rho}{r_{x}}\right)^\gamma \ge 2 |\nu \cdot n_x |  \right\}.
\label{def_blue}
\end{equation}
If $\nu$ is a unit vector and $\rho\in[0,r_{x})$, then the following hold.  If $\nu\cdot n_x \ge 0$ and $P_x+\rho\nu \not \in S_x$, then $P_x + \rho\nu \in \tilde\Omega$. If $\nu\cdot n_x \le 0$ and $P_x+\rho\nu \not \in S_x$, then $P_x + \rho\nu \in \Rm^2\setminus \tilde\Omega$.
\label{lemma:geometric}
\end{lemma}

In particular, $\partial \tilde\Omega\cap B_x \subseteq S_x$ (see Figure \ref{claim}).

\begin{figure}[htbp]
\begin{center}
\includegraphics[scale=0.8]{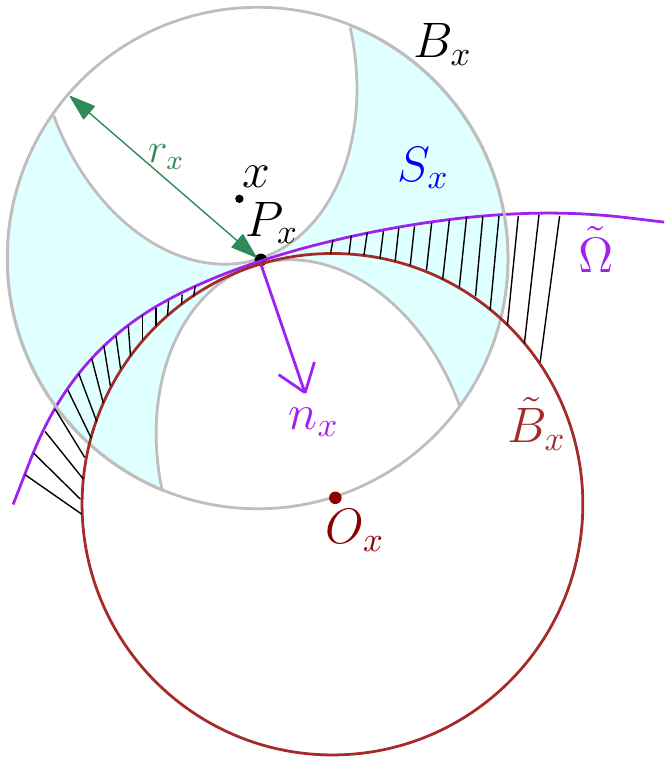}
\caption{The sets $S_x$ (shaded), $\tilde B_x$, and $\tilde\Omega \triangle \tilde B_x$ (lined).
\label{claim}}
\end{center}
\end{figure}

\begin{proof}
We only prove the first statement, as the proof of the second is analogous.  Let us assume $\nu\cdot n_x\ge 0$ and $P_x + \rho \nu \not\in \tilde\Omega$, with $|\nu|=1$ and $\rho\ge 0$. Then
\[
\nabla \tilde \varphi(P_x) \cdot \nu\ge 0 \qquad\text{and}\qquad \tilde\varphi(P_x + \rho \nu)\le 0,
\]
so we must have
$\nabla \tilde \varphi(P_x) \cdot \nu \leq \tilde A_\gamma \rho^\gamma$ because $\tilde\varphi(P_x)=0$.  Thus
\[
2\nu \cdot n_x \leq \frac{ 2\tilde A_\gamma \rho^\gamma} {|\nabla \tilde \varphi(P_x)|} = \left(\frac{\rho}{r_{x}} \right)^\gamma,
\]
so either $\rho\ge r_x$ or $P_x+\rho\nu  \in S_x$.
\end{proof}


\begin{proof}[Proof of Lemma \ref{lemma:d2vx}.]
Let $n_x,S_x$ be from Lemma \ref{lemma:geometric} and let $\tilde B_x := B(O_x, r_{x})$, where
\[
O_x :=P_x + r_{x}n_x.
\]
Then $P_x \in \partial \tilde B_x$ and the unit inner normal to $\partial \tilde B_x$ at $P_x$ is $n_x$.
We have $\partial \tilde B_x \cap B_x \subseteq S_x$ because if $P_x+\rho\nu\in \partial \tilde B_x \cap B_x$
with $|\nu|=1$ and $\rho> 0$, then
\[
r_x>\rho=2 r_x|\nu\cdot n_x|.
\]
Combining this with Lemma \ref{lemma:geometric} directly yields
\begin{equation}
(\tilde\Omega\triangle \tilde B_x)\cap  B_x \subseteq S_x,
\label{eq:in_S}
\end{equation}
with
\[
\tilde\Omega\triangle \tilde B_x:=(\tilde\Omega \backslash \tilde B_x)\cup (\tilde B_x \backslash \tilde\Omega)
\]
the symmetric difference of $\tilde\Omega$ and $\tilde B_x$ (the lined region in Figure \ref{claim}).
Let
\[
u_{\tilde B_x}(z): = \int_{\tilde B_x} \frac{(z-y)^{\perp}}{|z-y|^2} dy = 2\pi (\nabla^\perp \Delta^{-1} \chi_{\tilde B_x})(z)
\]
be the velocity field corresponding to the disc $\tilde B_x$. When  $|z-O_x| > r_{x}$, we have
by the rotational invariance of $u_{\tilde B_x}$ (and with $n$ the outer unit normal vector to $\partial B(O_x, |z-O_x|)$)
 \[
\begin{split}
u_{\tilde B_x}(z) &= \frac{(z-O_x)^\perp}{|z-O_x|} \left|u_{\tilde B_x}(z)\right|
\\ & =
 \frac{(z-O_x)^\perp}{|z-O_x|} \fint_{\partial B(O_x, |z-O_x|)} n\cdot 2\pi  \nabla \Delta^{-1} \chi_{\tilde B_x} d\sigma \\
&= \frac{(z-O_x)^\perp}{ |z-O_x|^2 } \int_{B(O_x, |z-O_x|)}  \chi_{\tilde B_x}(y) dy
\\ & =  \pi r_{x}^2 \frac{(z-O_x)^\perp}{|z-O_x|^2}.
\end{split}
\]
Differentiating this and noting that
\[
|x-O_x| = r_{x}+d(x)>r_x,
\]
yields
\begin{equation}
\label{d2uB}
|\nabla^2 u_{\tilde B_x}(x)| \leq
\frac{C}{r_{x}}.
\end{equation}
From the definitions of $\tilde v$ and $u_{\tilde B_x}$ we also have (with some $\tilde C<\infty$ and a new $C<\infty$)
\begin{equation}
|\nabla^2 \tv  (x)-\nabla^2 u_{\tilde B_x}(x)| \leq \underbrace{\int_{ \Rm^2 \setminus B_x} \frac{\tilde  C}{|x-y|^3}dy}_{\leq Cr_{x}^{-1}}  + \underbrace{ \int_{(\tilde\Omega\triangle \tilde B_x)\cap B_x} \frac{\tilde C}{|x-y|^3}dy }_{=: I}.
\label{eq:d2v_estimate}
\end{equation}
Finally, note that
\[
{\rm dist}(x,S_x)\geq \frac {d(x)}2.
\]
This holds because if
$P_x+\rho\nu\in  B(x, \frac 12{d(x)})$
with $|\nu|=1$ and $\rho\ge  0$, then
\[
\nu\cdot n_x\ge \cos \frac \pi 6>\frac 12
\qquad\text{and}\qquad
\rho\le \frac 32 d(x)<r_x
\]
(due to $d(x)\le \frac 14 r_x$), hence $P_x+\rho\nu\notin S_x$.
Also, if $|P_x-y| \geq 2d(x)$, then
\[
|P_x-y| \leq |x-y| + d(x) \leq 2 |x-y|.
\]
From these,  \eqref{eq:in_S}, and $|\theta|\le 2|\sin\theta|$ for  $|\theta|\le \frac\pi 2$
we now have
\[
\begin{split}
 I &\leq  \int_{S_x} \frac{\tilde C}{|x-y|^3}dy\\
 &\leq \int_{S_x\backslash B(P_x, 2d(x))}  \frac{\tilde C}{|x-y|^3}dy + \tilde C \left(\frac{d(x)}{2}\right)^{-3} \big|S_x\cap B(P_x, 2d(x))\big| \\
  &\leq \int_{S_x\backslash B(P_x, 2d(x))}  \frac{8\tilde C}{|P_x-y|^3}dy +  \frac{8 \tilde C}{d(x)^3} \big|S_x\cap B(P_x, 2d(x))\big| \\
 &\leq  \int_{2d(x)}^{r_{x}} \frac{8\tilde C}{\rho^{3}} 4 \left(\frac{\rho}{r_{x}}\right)^{\gamma} \rho \, d\rho +
 \frac{8 \tilde C}{d(x)^3} \int_0^{2d(x)} 4 \left(\frac{\rho}{r_{x}}\right)^{\gamma} \rho \, d\rho
\\& \leq C_\gamma d(x)^{-1+\gamma} r_{x}^{-\gamma}.
 \end{split}
 \]
This, \eqref{eq:d2v_estimate},  and \eqref{d2uB} now yield
\[
 |\nabla^2 \tv  (x)| \leq C_\gamma d(x)^{-1+\gamma} r_{x}^{-\gamma} + C r_x^{-1},
 \]
 so the result follows from  $d(x) \leq \frac{1}{4} r_{x}$.
 \end{proof}

\section{Finite time blow-up for small $\alpha>0$}\label{sec:sing}

In this section we prove Theorem \ref{main1234}, which is an immediate corollary of Theorem~\ref{T.6.1} below.

Let $\alpha\in(0,\frac 1{24})$ and $\epsilon>0$ be a small $\alpha$-dependent number, to be determined later.
Let $D^+:=\Rm^+\times\Rm^+$, $\Omega_1:=(\eps,4)\times(0,4)$, $\Omega_2:=(2\eps,3)\times(0,3)$, and
let $\Omega_0 \subseteq D^+$ be an open set whose boundary is a smooth simple closed curve and
which satisfies~$\Omega_2 \subseteq \Omega_0 \subseteq \Omega_1$.  Let $\omega$
be the  unique $H^3$ patch solution to \eqref{sqg}-\eqref{eq:velocity_law} with the initial data
\begin{equation}
\omega(\cdot,0) := \chi_{\Omega_0} - \chi_{\tilde{\Omega}_0}
\label{def:omega_0}
\end{equation}
and the maximal time of existence $T_\omega>0$.
Here, $\tilde{\Omega}_0$ is the reflection of $\Omega_0$ with respect to the $x_2$-axis.
Then oddness of $\omega_0$ in $x_1$ and the local uniqueness of the solution imply that
\begin{equation} \label{6.0}
\omega(\cdot, t) = \chi_{\Omega(t)} - \chi_{\tilde{\Omega}(t)}
\end{equation}
for $t\in[0,T_\omega)$, with $\Omega (t):=\Phi_t(\Omega_0)$  and $\tilde{\Omega}(t)$
the reflection of $\Omega(t)$ with respect to the~$x_2$-axis.  Note that $\Omega(t)$ is well-defined due to
Theorem~\ref{T.1.1bis}(a) and $H^3(\mathbb T)\subseteq C^{1,1}(\mathbb T)$).  We will show that $T_\omega<\infty$, that is, $\omega$ becomes singular in finite time.

More specifically, let
\begin{equation}\label {6.1}
T:=50(3\eps)^{2\alpha} \qquad\text{and}\qquad X(t):= \left[(3\eps)^{2\alpha}- \frac t{50} \right]^{1/2\alpha} \quad\text{for $t\in[0,T]$,}
\end{equation}
so that
\begin{equation}\label{aug2110}
X'(t)=-\frac 1{100\alpha} X(t)^{1-2\alpha},
\end{equation}
on $[0,T]$, with $X(0)=3\eps$ and $X(T)=0$, and let
\begin{equation}\label {6.2}
K(t):=\{ x\in D^+ \,:\, x_1\in(X(t),2) \text{ and } x_2\in(0,x_1)\}
\end{equation}
for $t\in[0,T]$.  We will show that if $T_\omega>T$, then $K(t)\subseteq \Omega(t)$ for all $t\in[0,T]$.  This yields a contradiction because then $\Omega(T)$ and $\tilde\Omega(T)$ touch at the origin (and thus they also cannot remain $H^3$).
\begin{figure}[htbp]
\begin{center}
\includegraphics[scale=1]{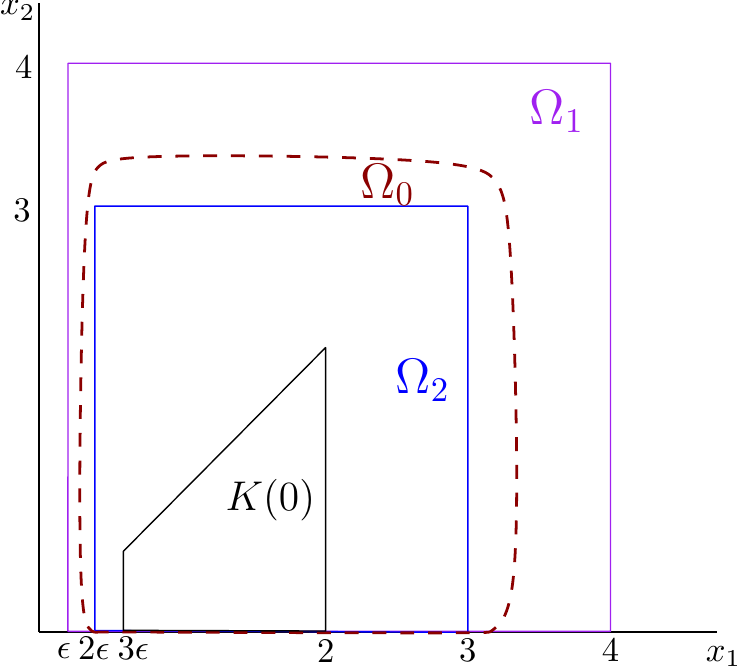}
\end{center}
\caption{The domains $\Omega_1, \Omega_2, \Omega_0$, and $K(0)$ (with $\omega_0=\chi_{\Omega_0}-\chi_{\tilde  \Omega_0}$). \label{fig:def_domains}}
\end{figure}

This result will, in fact, hold for the less regular $C^{1,\gamma}$ patches, but in this case we need to assume oddness of $\omega$ in $x_1$  (this is not immediate from the same property of $\omega_0$ without knowing local uniqueness in this class).  Before we can prove the result, however, we  need to obtain some estimates on
the velocity $u$, the most crucial of which is Proposition \ref{P.6.3}.

{\it Remark.} The fact that the fraction on the right-hand side of (\ref{aug2110}) blows-up as $\alpha\to 0$ may seem
worrying but   $\eps$ will go to zero quickly as $\alpha\to 0$ (and $X(t)\in[0,3\eps]$), so this growth will be
compensated by the term $X(t)^{1-2\alpha}$ which decays as $\alpha\to 0$.

\subsection{Some estimates on the velocity fields}
\label{sec:estimates}

Let us start with some basic estimates on the fluid velocities for a general $\omega$.

\begin{lemma}\label{lemma:uniform_u_bound}
For $\alpha\in(0,\frac 12)$ and $u(\cdot,t)$ as in  \eqref{eq:velocity_law} with $\omega(\cdot,t) \in L^1(D) \cap L^\infty(D)$, we have
\begin{equation}\label{uLinfty}
\|u(\cdot,t)\|_{L^\infty} \leq \frac{2\pi}{1-2\alpha} \|\omega(\cdot, t)\|_{L^\infty}+ 2\|\omega(\cdot,t)\|_{L^1}
\end{equation}
and
\begin{equation}\label{uHold}
\|u(\cdot,t)\|_{C^{1-2\alpha}} \leq  \frac {8\pi}{\alpha(1-2\alpha)} \|\omega(\cdot, t)\|_{L^\infty} + 2\|\omega(\cdot,t)\|_{L^1}.
\end{equation}
Furthermore, if $\omega$ is weak-$*$ continuous as an $L^\infty(D)$-valued function on a time interval~$[a,b]$,
and is supported inside a fixed compact subset of $\bar D$ for every $t\in [a,b]$, then $u$ is continuous on $\bar D \times [a,b]$.
\end{lemma}

\begin{proof}
Let  $\eta:\mathbb{R}^2\to\mathbb{R}$ be the odd extension of $\omega(\cdot,t)$ to
the whole plane.
The Bio-Savart law \eqref{eq:velocity_law} for $x\in D$  then becomes
\begin{equation}\label{genulaw}
u(x,t) =  \int_{\mathbb{R}^2} \frac{(x-y)^\perp}{|x-y|^{2+2\alpha}} \eta(y) dy,
\end{equation}
and \eqref{uLinfty} follows from
\begin{equation*}
\begin{split}
|u(x,t)| &\leq  \int_{|x-y|\leq 1} \frac{|\eta(y)|}{|x-y|^{1+2\alpha}}dy
+ \int_{|x-y|> 1} \frac{|\eta(y)|}{|x-y|^{1+2\alpha}} dy\\
&\leq  \|\eta\|_{L^\infty} \int_{|x-y|\leq 1}
\frac{1}{|x-y|^{1+2\alpha}}  dy +
\|\eta\|_{L^1}
\\ & \leq \frac{2\pi}{1-2\alpha} \|\omega(\cdot,t)\|_{L^\infty}+ 2\|\omega(\cdot,t)\|_{L^1}.
\end{split}
\end{equation*}
To prove \eqref{uHold}, consider any $x,z \in \bar D$ with $r:=|x-z|$.
Then
\begin{equation*}
\begin{split}
|u(x,t)-u(z,t)| \leq & \int_{B(x,2r)} \frac{1}{|x-y|^{1+2\alpha}} \eta(y)\,dy + \int_{B(x,2r)} \frac{1}{|z-y|^{1+2\alpha}} \eta(y)\,dy \\
& + \int_{\Rm^2 \setminus B(x,2r)} \left| \frac{(x-y)^\perp}{|x-y|^{2+2\alpha}} - \frac{(z-y)^\perp}{|z-y|^{2+2\alpha}} \right| \eta(y)\,dy  \\
\le & 4\pi\|\eta\|_{L^\infty} \int_0^{3r} s^{-2\alpha}\,ds + 32\|\eta\|_{L^\infty}\int_{2r}^\infty r s^{-1-2\alpha}\,ds \\
\leq & \left( \frac {12\pi}{1-2\alpha}+\frac {32}{2\alpha} \right) \|\eta\|_{L^\infty} |x-z|^{1-2\alpha}.
\end{split}
\end{equation*}
Combining this with \eqref{uLinfty} yields \eqref{uHold}.


It remains to prove the last claim.  Since the kernel in \eqref{genulaw} is $L^1$ on any compact subset of $\bar D$, the assumptions show that $u$ is continuous in $t\in [a,b]$  for any fixed $x\in\bar D$.  The claim now follows from uniform continuity of $u$ in $x\in\bar D$, see \eqref{uHold}.
\end{proof}

For $y=(y_1,y_2)\in \bar D^+ = \Rm^+\times\Rm^+$, we denote $\bar y:=(y_1,-y_2)$
and $\tilde y:=(-y_1,y_2).$
If~$\omega(\cdot, t)\in L^\infty (D)$ is odd in $x_1$, then \eqref{eq:velocity_law} becomes (we drop $t$ from the notation in this sub-section)
\begin{align}
u_1(x) &= -\int_{D^+} K_1(x,y) \omega(y) dy, \label{def:u1}\\
u_2(x) &= \int_{D^+} K_2(x,y) \omega(y) dy, \notag
\end{align}
where
\begin{equation}
K_1(x,y) =
\underbrace{\frac{y_2-x_2}{|x-y|^{2+2\alpha}}}_{K_{11}(x,y)} -
\underbrace{\frac{y_2-x_2}{|x-\tilde y|^{2+2\alpha}}}_{K_{12}(x,y)} -
 \underbrace{\frac{y_2+x_2}{|x+y|^{2+2\alpha}}}_{K_{13}(x,y)} +
\underbrace{\frac{y_2+x_2}{|x-\bar y|^{2+2\alpha}}}_{K_{14}(x,y)},
\label{def:K1}
\end{equation}
\begin{equation}
K_2(x,y) =
\underbrace{\frac{y_1-x_1}{|x-y|^{2+2\alpha}}}_{K_{21}(x,y)} +
\underbrace{\frac{y_1+x_1}{|x-\tilde y|^{2+2\alpha}}}_{K_{22}(x,y)} -
\underbrace{\frac{y_1+x_1}{|x+y|^{2+2\alpha}}}_{K_{23}(x,y)} -
\underbrace{\frac{y_1-x_1}{|x-\bar y|^{2+2\alpha}}}_{K_{24}(x,y)}.
\label{def:K2}
\end{equation}
Let us start
with some simple observations about $K_1$ and $K_2$.
\begin{lemma}
For $\alpha\in(0,\frac 12)$ and  $x,y\in D^+$ we have the following:
\\[0.2cm]
(a) $K_1(x,y) \geq K_{11}(x,y) - K_{12}(x,y)$.
\\[0.1cm]
(b) ${\rm sgn}(y_2-x_2) (K_{11}(x,y) - K_{12}(x,y))\ge 0$.
\\[0.1cm]
(c) $K_2(x,y) \geq K_{21}(x,y) - K_{24}(x,y)$.
\\[0.1cm]
(d) ${\rm sgn}(y_1-x_1) (K_{21}(x,y) - K_{24}(x,y)) \geq 0$.
\label{lemma:observation_u1}
\end{lemma}

\begin{proof}
Part (a) is immediate from $|x-\bar y|\le |x+y|$ and (b) from $|x-y| \leq |x-\tilde y|$.
Exchanging $\bar y$ and $\tilde y$ yields the proofs of (c) and (d).
%
%
%
\end{proof}

Our goal will be to show that if the solution with the initial data from \eqref{def:omega_0} exists globally ,
in which case $0\le\omega\le 1$ on $D^+$ by symmetry, then the patch $\Omega(t)$ and its reflection across the $x_2$
axis must touch at the origin in finite time, which is a contradiction. In particular,  we will need to show that $u_1$
is sufficiently negative in an appropriate subset of $D^+$ (at least for some time).   We will do this by separately estimating the ``bad'' part
\[
u_{1}^{bad}(x) := - \int_{\Rm^+\times(0,x_2)} K_1(x,y) \omega(y) dy
\]
 of the integral in \eqref{def:u1}, (where $K_{11}-K_{12}<0$) and the ``good'' part
 \[
u_{1}^{good}(x) := - \int_{\Rm^+\times(x_2,\infty)} K_1(x,y) \omega(y) dy
\]
(where $K_{11}-K_{12}\ge 0$).   We will also obtain similar estimates for the $u_2$ analogs
\[
u_{2}^{bad}(x) := \int_{(0,x_1)\times \Rm^+} K_2(x,y) \omega(y) dy,
\]
\[
u_{2}^{good}(x) := \int_{(x_1,\infty)\times \Rm^+} K_2(x,y) \omega(y) dy.
\]


\begin{lemma}
Let $\alpha\in(0,\frac 12)$ and assume that $\omega$ is odd in $x_1$ and $0\le \omega\le 1$ on $D^+$.
\\[0.2cm]
(a) If $x\in \overline{D^+}$ and $x_2\leq x_1$, then
\[
u_1^{bad}(x) \leq \frac{1}{\alpha}\left(\frac{1}{1-2\alpha} -2^{-\alpha}\right) x_1^{1-2\alpha}.
\]
\\[0.1cm]
(b) If $x\in \overline{D^+}$ and $x_2\geq x_1$, then
\[
u_2^{bad}(x) \geq - \frac{1}{\alpha}\left(\frac{1}{1-2\alpha} -2^{-\alpha}\right) x_2^{1-2\alpha}.
\]
\label{lemma:bad_parts_u1}
\end{lemma}

\begin{proof}
(a) As $0\le\omega\le 1$ on $D^+$, it follows from Lemma \ref{lemma:observation_u1}(a,b)   that
\begin{equation*}
\begin{split}
u_1^{bad}(x)
&\le  -\int_{\Rm^+\times(0,x_2)} \left(\frac{y_2-x_2}{|x-y|^{2+2\alpha}} - \frac{y_2-x_2}{|x-\tilde y|^{2+2\alpha}}\right) \omega(y) dy\\
&\le  -\int_{\Rm^+\times(0,x_2)} \left(\frac{y_2-x_2}{|x-y|^{2+2\alpha}} - \frac{y_2-x_2}{|x-\tilde y|^{2+2\alpha}}\right)dy\\
&= -\int_{(0,2x_1)\times(0,x_2)} \frac{y_2-x_2}{|x-y|^{2+2\alpha}}  dy.
\end{split}
\end{equation*}
The equality holds due to identity
\[
\int_{\Rm^+\times(0,x_2)} \frac{y_2-x_2}{|x-\tilde y|^{2+2\alpha}} dy
= \int_{(2x_1,\infty)\times(0,x_2)} \frac{y_2-x_2}{|x-y|^{2+2\alpha}}dy,
\]
that can be verified by a change of variables $y_1\mapsto y_1+2x_1$.
Now, the change of variables~$z:=x-y$, symmetry, together with the assumption $x_2\le x_1$, yield
\begin{equation}
\begin{split}
u_1^{bad} (x)
&\le  2\int_{(0,x_1)\times (0,x_2)} \frac{z_2}{(z_1^2 + z_2^2)^{1+\alpha}}dz
\\ &= \frac{1}{\alpha} \int_0^{x_1}\left( \frac{1}{z_1^{2\alpha}} -
  \frac{1}{(z_1^2 + x_2^2)^\alpha} \right)dz_1\\
&\leq  \frac{1}{\alpha(1-2\alpha)} x_1^{1-2\alpha} - \frac{1}{\alpha} \int_0^{x_1}  \frac{1}{(2 x_1^2)^\alpha} dz_1
\\ & =  \frac{1}{\alpha}\left(\frac{1}{1-2\alpha} -2^{-\alpha}\right) x_1^{1-2\alpha}.
\end{split}
\end{equation}
The proof of part (b) is analogous to (a).
\end{proof}

%

In the estimate of the ``good'' parts of $u_1,u_2$ we will in addition assume that for some~$x\in D^+$ we have $\omega=1$ on the triangle
\begin{equation}
A(x) := \left\{ y\,:\,  y_1\in \left(x_1,x_1+1 \right) \text{ and } y_2 \in(x_2, x_2 + y_1 - x_1) \right\},
\label{eq:def_A}
\end{equation}
which is depicted in Figure \ref{fig:def_A}.  This assumption will feature in the proof of the comparison-principle-type result $K(t)\subseteq \Omega(t)$ (mentioned above) in the next sub-section.

\begin{figure}[htbp]
\begin{center}
\includegraphics[scale=1]{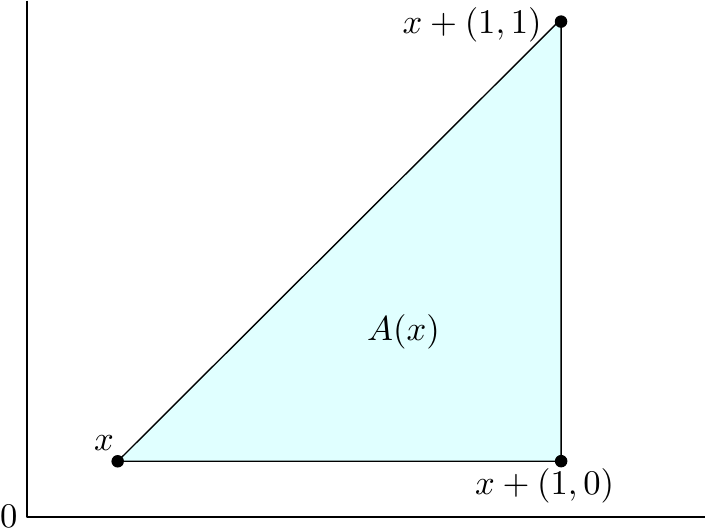}
\end{center}
\caption{The domain $A(x)$. \label{fig:def_A}}
\end{figure}

\begin{lemma}
\label{lemma:u1_good}
Let $\alpha\in(0,\frac 12)$ and assume that $\omega$ is odd in $x_1$ and for some $x\in  \overline{D^+}$
we have $\omega\ge \chi_{A(x)}$ on $D^+$, with $A(x)$ from \eqref{eq:def_A}.
There exists $\delta_\alpha\in(0,1)$, depending only on $\alpha$, such that the following hold.
\\[0.2cm]
(a) If $x_1\le \delta_\alpha$, then
\begin{equation*}
u_1^{good}(x) \leq -\frac{1}{6\cdot 20^\alpha \alpha}x_1^{1-2\alpha}.
\end{equation*}
\\[0.1cm]
(b) If $x_2\le \delta_\alpha$, then
\begin{equation*}
u_2^{good}(x) \geq \frac{1}{5\cdot 8^\alpha \alpha}x_2^{1-2\alpha}.
\end{equation*}
\end{lemma}

\begin{proof}
(a) Using Lemma \ref{lemma:observation_u1}(a) and then changing variables $y_1\mapsto y_1+2x_1$, we obtain
\begin{equation}
\begin{split}
u_1^{good}(x) &\leq - \int_{A(x)}
\left(\frac{y_2-x_2}{|x-y|^{2+2\alpha}} -  \frac{y_2-x_2}{|x-\tilde y|^{2+2\alpha}}\right) dy \\
&= -\int_{A(x)} \frac{y_2-x_2}{|x-y|^{2+2\alpha}} dy +
\int_{A(x)+2x_1e_1} \frac{y_2-x_2}{|x-y|^{2+2\alpha}} dy,
\end{split}
\end{equation}
with $e_1:=(1,0)$.
Since the last two integrands are the same, after a cancellation due to the opposite signs we obtain
\[
u_1^{good}(x) \leq -\underbrace{\int_{A_1}\frac{y_2-x_2}{|x-y|^{2+2\alpha}}  dy}_{T_1}
+ \underbrace{ \int_{A_2}\frac{y_2-x_2}{|x-y|^{2+2\alpha}}  dy}_{T_2},
\]
with the domains
\begin{equation*}
\begin{split}
A_1 &:= \left\{ y \,:\, y_2 \in \left(x_2, x_2 + 1 \right) \text{ and } y_1\in \left(x_1+y_2-x_2, 3x_1+y_2-x_2 \right) \right\},\\
A_2 &:=  \left(x_1+1, 3x_1+1\right) \times \left(x_2, x_2 + 1 \right)
\end{split}
\end{equation*}
illustrated in Figure \ref{fig:def_A1_A2}.
\begin{figure}[htbp]
\begin{center}
\includegraphics[scale=1]{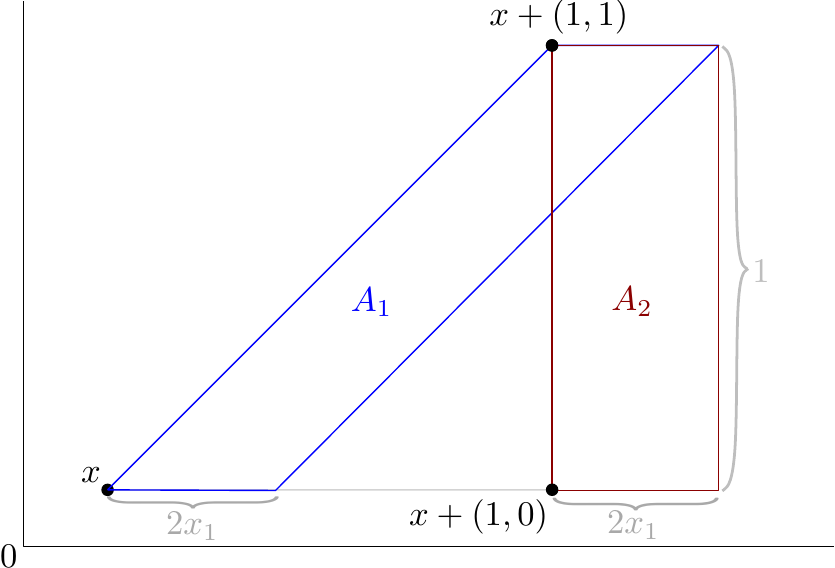}
\caption{The domains $A_1$ and $A_2$. \label{fig:def_A1_A2}}
\end{center}
\end{figure}
Since for $y\in A_2$ we have  $y_2-x_2\le 1 \le |x-y|$, we obtain
\[
T_2 \leq  |A_2|  =2x_1.
\]
To control $T_1$, we first note that its integrand is positive, so we can get a lower bound on~$T_1$ by only integrating
over $A_1':=A_1 \cap [\Rm\times(x_2+ 2x_1,\infty)]$. For $y\in A_1'$ we have
\[
y_2-x_2\ge \frac{1}{2}(y_1-x_1),
\]
which yields
\[
\sqrt{5}(y_2-x_2) \geq |x-y|.
\]
This gives
\begin{equation}
\begin{split}
T_1 &\geq
5^{-1-\alpha}  \int_{A_1'}  (y_2-x_2)^{-(1+2\alpha)}  dy
\\ &=   5^{-1-\alpha} 2x_1 \int_{x_2+2x_1}^{x_2+1}
(y_2-x_2)^{-(1+2\alpha)}  dy_2 \\
&=   \frac{1}{5^{1+\alpha}\alpha} x_1 [(2x_1)^{-2\alpha} - 1].
\end{split}
\end{equation}
Putting the estimates for $T_1$ and $T_2$ together yields
\begin{equation}
u_1^{good}(x) \leq
- \left[ \frac{1}{5\cdot 20^\alpha \alpha} -  \left( \frac{1}{5^{1+\alpha} \alpha} + 2\right) x_1^{2\alpha}  \right] x_1^{1-2\alpha}.
\end{equation}
The result now follows for some small enough $\delta_\alpha>0$.

(b) Using Lemma \ref{lemma:observation_u1}(c) and then the change of variables $y_2\mapsto y_2+2x_2$, we obtain
\begin{equation}
\begin{split}
u_2^{good}(x) &\geq  \int_{A(x)}
\left(\frac{y_1-x_1}{|x-y|^{2+2\alpha}} - \frac{y_1-x_1}{|x-\bar y|^{2+2\alpha}}\right) dy \\
&= \int_{A(x)} \frac{y_1-x_1}{|x-y|^{2+2\alpha}} dy -
\int_{A(x)+2x_2e_2} \frac{y_1-x_1}{|x-y|^{2+2\alpha}} dy,
\end{split}
\end{equation}
with $e_2:=(0,1)$.
Since the last two integrands are the same, after a cancellation due to the opposite signs we obtain
\[
u_2^{good}(x) \geq \int_{B_1} \frac{y_1-x_1}{|x-y|^{2+2\alpha}}  dy
- \int_{B_2} \frac{y_1-x_1}{|x-y|^{2+2\alpha}}  dy,
\]
with the domains
\begin{equation*}
\begin{split}
B_1 &:= \left(x_1, x_1 + 1 \right) \times \left(x_2, 3x_2\right), \\
B_2 &:= \left\{ y \,:\, y_1 \in \left(x_1, x_1 + 1 \right) \text{ and } y_2\in \left(x_2+y_1-x_1, 3x_2+y_1-x_1 \right) \right\}
\end{split}
\end{equation*}
illustrated in Figure \ref{fig:def_B1_B2}.
\begin{figure}[htbp]
\begin{center}
\includegraphics[scale=1]{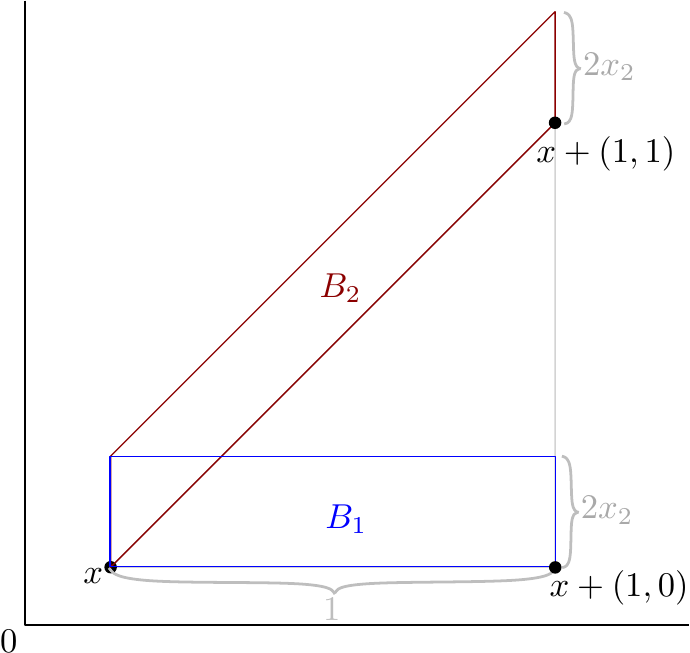}
\caption{The domains $B_1$ and $B_2$. \label{fig:def_B1_B2}}
\end{center}
\end{figure}
The change of variables $y_2\mapsto y_2-(y_1-x_1)$ in the second integral then yields
\[
u_2^{good}(x) \geq \int_{B_1} \left( \frac{y_1-x_1}{|x-y|^{2+2\alpha}} - \frac{y_1-x_1}{|x-(y_1,y_2+y_1-x_1)|^{2+2\alpha}} \right) dy.
\]
Since the integrand is positive, and for $y\in (x_1+2x_2, x_1 + 1 ) \times (x_2, 3x_2)$ we have
\[
|x-(y_1,y_2+y_1-x_1)|^2 =2|x-y|^2+(y_2-x_2)[2(y_1-x_1)-(y_2-x_2)] > 2|x-y|^2
\]
due to $y_1-x_1> 2x_2>y_2-x_2>0$, it follows that
\[
u_2^{good}(x) \geq \left (1-2^{-1-\alpha}\right) \int_{(x_1+2x_2, x_1 + 1 ) \times (x_2, 3x_2)}  \frac{y_1-x_1}{|x-y|^{2+2\alpha}}  dy.
\]
On this domain of integration we have $y_1-x_1\ge \frac 1{\sqrt 2} |x-y|$, so
\begin{equation}
\begin{split}
u_2^{good}(x)
&\geq  2^{-2-2\alpha} \int_{(x_1+2x_2, x_1 + 1 ) \times (x_2, 3x_2)} (y_1-x_1)^{-1-2\alpha} dy \\
&= 2^{-1-\alpha} x_2 \int_{x_1+2x_2}^{x_1+1} (y_1-x_1)^{-1-2\alpha} dy_1  = \frac{1}{2^{2+\alpha} \alpha}  x_2 [(2x_2)^{-2\alpha} - 1] \\
&=  \left[ \frac{1}{4\cdot 8^{\alpha} \alpha} - \frac{1}{2^{2+\alpha} \alpha} x_2^{2\alpha} \right] x_2^{1-2\alpha}.
\end{split}
\end{equation}
The result now follows for some small enough $\delta_\alpha>0$.
\end{proof}

The last two lemmas combine to the following result for small $\alpha$.

\begin{proposition}
\label{P.6.3}
Let $\alpha\in(0,\frac 1{24})$ and assume that $\omega$ is odd in $x_1$ and for some $x\in  \overline{D^+}$ we have $\chi_{A(x)}\le \omega\le 1$ on $D^+$, with $A(x)$ from \eqref{eq:def_A}.
Then there exists $\delta_\alpha\in(0,1)$, depending only on $\alpha$, such that the following hold.
\\[0.2cm]
(a) If $x_2\le x_1\le \delta_\alpha$, then
\begin{equation*}
u_1(x) \leq -\frac{1}{50 \alpha}x_1^{1-2\alpha}.
\end{equation*}
\\[0.1cm]
(b) If $x_1\le x_2\le \delta_\alpha$, then
\begin{equation*}
u_2(x) \geq \frac{1}{50 \alpha}x_2^{1-2\alpha}.
\end{equation*}
\end{proposition}

\begin{proof}
(a)  This is immediate from the last two lemmas and $u_1=u_1^{bad}+u_1^{good}$
because
\[
-\frac 1{6\cdot 20^\alpha} + \left(\frac{1}{1-2\alpha} -2^{-\alpha}\right)
\]
is increasing in $\alpha$ and its value for $\alpha= \frac 1{24}$ is less than $-1/{50}$.

(b)  Since $5\cdot 8^\alpha < 6\cdot 20^\alpha$, this is analogous to (a).
\end{proof}

\subsection{The finite time singularity analysis}
\label{sec:blowup}

 Let us now return to the setting from the beginning of this section.
The initial condition  we consider  is odd in $x_1$, and  the resulting unique~$H^3$ patch solution is also odd.
We will run the blow-up argument in the class of the less regular $C^{1,\gamma}$ patch solutions to \eqref{sqg}-\eqref{eq:velocity_law}, and
show that any such solution either has a finite maximal time of existence (i.e., loss of existence) or stops being odd (i.e., loss of uniqueness).  Of course, the latter cannot happen for the $H^3$ patch solution.

\begin{theorem} \label {T.6.1}
Let $\alpha\in(0,\frac 1{24})$ and $\eps>0$ be small enough.  Let $\omega(\cdot,0)$ be
given by~\eqref{def:omega_0}, with a bounded open $\Omega_0\subseteq D^+$ such that
$(2\eps,3)\times(0,3)\subseteq \Omega_0 \subseteq (\eps,4)\times(0,4)$ and $\partial\Omega_0$ is a smooth simple closed curve.
Then for any $\gamma> \frac{2\alpha}{1-2\alpha}$, there is no odd-in-$x_1$ $C^{1,\gamma}$ patch solution $\omega$ to \eqref{sqg}-\eqref{eq:velocity_law} on any interval $[0,T')$ with $T'>50(3\eps)^{2\alpha}$.
\end{theorem}

This immediately yields Theorem \ref{main1234} because the (local) $H^3$ solution for this initial condition
is odd in $x_1$ (due to of its uniqueness), and it is  $C^{1,\gamma}$ for each $\gamma\in(0,1]$.

\begin{proof}
Let us assume that such a solution exists and let  $T, X(t),K(t)$ be from \eqref{6.1}-\eqref{6.2}. The solution then has the form \eqref{6.0}, and we will show that
$K(t)\subseteq \Omega(t)$ for each $t\in[0,T]$.  This is a contradiction  because
then the patches $\Omega(T)$ and $\tilde\Omega(T)$ touch at $0$. 

As $|\Omega(t)|=|\Omega_0|\le 16$, Lemma \ref{lemma:uniform_u_bound}  implies
\begin{equation} \label{6.4}
\|u(\cdot,t)\|_{L^\infty}\le 100
\end{equation}
for all $t\in[0,T]$.  Since $\partial\Omega(t)$ is continuous in $t\in[0,T]$ with respect to
the Hausdorff distance of sets, the lemma  also shows that $u$ is continuous on $\bar D\times[0,T]$.

Consider $\delta_\alpha\in(0,1)$ from Proposition \ref{P.6.3} and let the constant $\eps$ in \eqref{6.1} satisfy
\[
\eps\le \frac{\delta_\alpha^{1/2\alpha}}{3\cdot 100^{1/\alpha}}.
\]
We know from \eqref{6.4}  that the function $f(t):={\rm dist}(D^+\setminus\overline{\Omega(t)},K(t))$ is continuous on~$[0,T]$.
Hence, if $K(t)$ is not contained in $\Omega(t)$ at some $t\in[0,T]$, then there is the first
time~$t_0\in[0,T]$ such that $f(t_0)=0$.  As $f(0)\ge\eps>0$, we have~$t_0>0$ and~$K(t_0)\subseteq \Omega(t_0)$.

Let us assume that such $t_0$ exists and let
\[
\Omega_3:=(\delta_\alpha,\frac 52)\times(0,\frac 52).
\]
Then $T= 200^{-1}\delta_\alpha$, the estimate \eqref{6.4}, and $ 2\eps< \frac 12\delta_\alpha< \frac 12$ imply
\[
[D^+\setminus\overline{\Omega(t_0)}] \cap \Omega_3=\emptyset,
\]
where we also used that symmetry and Theorem~\ref{T.1.1bis} yield
\begin{equation} \label{9.9}
D^+\setminus\overline{\Omega(t)} = \Phi_{t}(D^+\setminus \overline{\Omega(0)})
\end{equation}
 for any $t\in[0,T]$.
As $t_0$ is the first time with $f(t_0)=0$, it follows that there exists some
\begin{equation} \label{6.5}
x\in \partial[D^+\setminus\overline{\Omega(t_0)}] \cap [I_1\cup I_2],
\end{equation}
where $I_1=\{X(t_0)\}\times[0,X(t_0))$  and $I_2$ is the closed straight segment connecting the points $(X(t_0),X(t_0))$
and $(\delta_\alpha,\delta_\alpha)$ (see Figure~\ref{fig:t1}).
\begin{figure}[htbp]
\begin{center}
\includegraphics[scale=1.2]{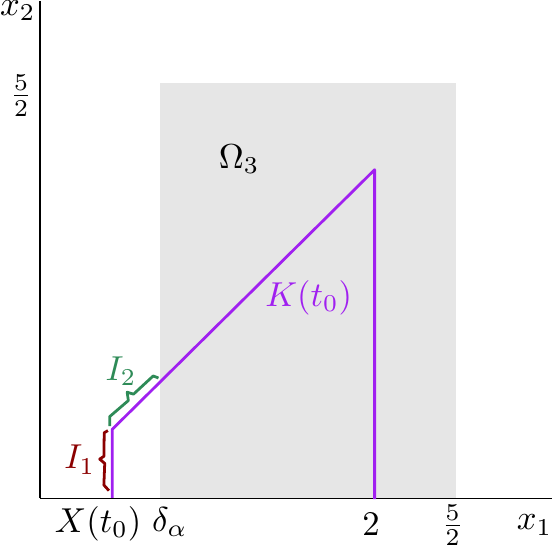}
\caption{The segments $I_1$ and $I_2$ and the sets $\Omega_3$ and $K(t_0)$.\label{fig:t1}}
\end{center}
\end{figure}

If $x\in I_1$, then the triangle $A(x)$ defined in \eqref{eq:def_A} and
depicted in Figure~\ref{fig:def_A} satisfies
\[
A(x)\subseteq K(t_0)\subseteq \Omega(t_0)
\]
because
\[
X(t_0)\le 3\eps<\delta_\alpha<1.
\]
Hence  Proposition \ref{P.6.3}(a) and $x_1=X(t_0)$ yield
\[
u_1(x,t_0)\le -\frac 1{50\alpha} x_1^{1-2\alpha} < -\frac 1{100\alpha} x_1^{1-2\alpha} = X'(t_0).
\]
Since
\[
\Phi_{t_0}(D^+\setminus \overline{\Omega(0)}) \cap B(x,r)\neq\emptyset
\]
for any $r>0$ and $u$ is continuous, it follows from this and \eqref{6.4}
that for any sufficiently small~ $s\in(0,\frac1{100}[X(t_0)-x_2])$ we have
\[
\Phi_{t_0-s}(D^+\setminus \overline{\Omega(0)})\cap [(X(t_0-s),2)\times(0,X(t_0))] \neq \emptyset.
\]
From \eqref{9.9} and $(X(t_0-s),2)\times(0,X(t_0))\subseteq K(t_0-s)$ we now obtain $f(t_0-s)=0$ for these $s$, a contradiction with the choice of $t_0$.

If now $x\in I_2$, so that $x_1=x_2\le\delta_\alpha$, a similar argument and Proposition \ref{P.6.3}(a,b) yield
\[
(-1)^{j-1}u_j(x,t_0)\le -\frac 1{50\alpha} x_1^{1-2\alpha} < -\frac 1{100\alpha} x_1^{1-2\alpha} \le  X'(t_0)
\]
for $j=1,2$, 
and thus
\[
\Phi_{t_0-s}(D^+\setminus \overline{\Omega(0)})\cap [(x_1+X(t_0-s)-X(t_0),2)\times(0,x_1-X(t_0-s)+X(t_0))] \neq \emptyset
\]
for all small enough $s>0$.  We again obtain a contradiction because $X(t_0)\le x_1=x_2$ implies $(x_1+X(t_0-s)-X(t_0),2)\times(0,x_1)\subseteq K(t_0-s)$.
\end{proof}

\end{document}